\documentclass[11pt,twoside]{article}
\usepackage{authblk}
\usepackage[scaled=0.8]{DejaVuSansMono}
\usepackage{emptypage}
\usepackage[utf8]{inputenc}
\usepackage[a4paper,marginratio={5:5,5:7},width=147mm]{geometry}

\usepackage{fancyhdr}
\newcommand{\maybebf}{}

\fancypagestyle{plain}{%
    \fancyhf{}%
    \fancyfoot[C]{\maybebf \thepage}
    
}
\fancypagestyle{intro}{
    \renewcommand{\sectionmark}[1]{\markboth{##1}{}}
    \fancyhf{}
    \fancyhead[LE,RO]{\leftmark}
    \fancyfoot[C]{\maybebf \thepage}

}

\renewcommand{\sectionmark}[1]{%
  \markboth{%
    \ifnum\value{section}>0
      \maybebf{\thesection.}\space
    \fi
    #1%
  }{%
    \ifnum\value{subsection}=0
            \thesection. #1
        \fi%
  }%
}
\renewcommand{\subsectionmark}[1]{%
    \markright{%
        \ifnum\value{subsection}>0
            \thesubsection. #1
        \fi%
    }%
}
\usepackage[english]{babel}
\usepackage{graphicx}
\usepackage{amsmath}
\usepackage{amssymb}
\usepackage{amsthm}
\usepackage{mathtools}
\usepackage{microtype}
\usepackage{tikz}
\usetikzlibrary{arrows.meta,arrows,decorations.pathmorphing,shapes,knots,hobby}

\usepackage{asymptote}
\usepackage{tikz-cd}
\usepackage{parskip}
\usepackage{float}
\usepackage{caption}
\usepackage{subcaption}
\usepackage{enumitem}
\usepackage{adjustbox}
\usepackage{xfakebold}
\usepackage{mdframed}
\usepackage{booktabs}
\usepackage{multirow}
\usepackage{siunitx}
\usepackage{nicematrix}
\NiceMatrixOptions{xdots/shorten=0.87em, renew-dots}

\usepackage[calc,showdow,english]{datetime2}
\DTMnewdatestyle{mydateformat}{%
}

\usepackage{comment}
\usepackage{import}
\usepackage{standalone}

\usepackage{xcolor}
\definecolor{blue1}{RGB}{41,41,255}
\definecolor{blue2}{RGB}{38,143,235}
\definecolor{blue3}{RGB}{34,214,202}

\usepackage{csquotes}
\usepackage[unicode]{hyperref}
\usepackage[doi=false,isbn=false,url=false,
maxbibnames=99,backend=biber,style=alphabetic,sorting=nyt]{biblatex}

\theoremstyle{plain}
\newtheorem{thm}{Theorem}[section]
\newtheorem{prop}[thm]{Proposition}
\newtheorem{lemma}[thm]{Lemma}

\newtheorem{conj}[thm]{Conjecture}
\newtheorem*{thm*}{Theorem}
\newtheorem*{prop*}{Proposition}
\newtheorem*{conj*}{Conjecture}
\newtheorem*{cor*}{Corollary}

\theoremstyle{definition}
\newtheorem{defin}[thm]{Definition}

\newtheorem*{defin*}{Definition}

\theoremstyle{remark}
\newtheorem{remark}[thm]{Remark}

\counterwithin{figure}{section}
\counterwithin{table}{section}

\newcommand{\defeq}{\coloneqq}

\newcommand{\neu}{\mathcal N}
\newcommand{\poly}{\mathcal P}
\newcommand{\polym}{\mathfrak P}
\newcommand{\linn}{\mathcal D}
\newcommand{\inj}{\hookrightarrow}
\newcommand{\surj}{\twoheadrightarrow}
\newcommand{\midx}[1]{{{#1}}}
\newcommand{\heqsep}{\hspace{0.8em}}

\newcommand{\id}{\mathrm{id}}
\newcommand{\ZZ}{\mathbb Z}
\newcommand{\QQ}{\mathbb Q}
\newcommand{\RR}{\mathbb R}
\newcommand{\CC}{\mathbb C}
\newcommand{\til}{\widetilde}
\newcommand{\blt}{*}

\DeclarePairedDelimiter\abs{\lvert}{\rvert}%
\DeclarePairedDelimiter\norm{\lVert}{\rVert}%

\DeclarePairedDelimiter\floor{\lfloor}{\rfloor}

\makeatletter
\let\oldabs\abs
\def\abs{\@ifstar{\oldabs}{\oldabs*}}
\let\oldnorm\norm
\def\norm{\@ifstar{\oldnorm}{\oldnorm*}}
\makeatother

\DeclareMathOperator{\image}{im}
\DeclareMathOperator{\coker}{coker}
\DeclareMathOperator{\Hom}{Hom}
\DeclareMathOperator{\ab}{ab}
\DeclareMathOperator{\rk}{rk}
\DeclareMathOperator{\pdim}{pdim}
\DeclareMathOperator{\tr}{tr}
\DeclareMathOperator{\ord}{ord}

\newcommand{\fbseries}{
\unskip\setBold\aftergroup\unsetBold\aftergroup\ignorespaces}
\makeatletter
\newcommand{\setBoldness}[1]{\def\fake@bold{#1}}
\makeatother

\tikzset{
    rots/.style={anchor=south, rotate=90, inner sep=.5mm}
}

\newcommand\restr[2]{{
  \left.\kern-\nulldelimiterspace 
  #1 
  \vphantom{\big|} 
  \right|_{#2} 
  }}

\author{Jacopo G. Chen\thanks{Scuola Normale Superiore, Pisa, Italy. Email: \href{mailto:jacopo.chen@sns.it}{\texttt{jacopo.chen@sns.it}}}}
\title{Computing the twisted $L^2$-Euler characteristic}

\begin{document}

\DTMsetdatestyle{mydateformat}
\date{}
\maketitle

\begingroup
\centering\small
\textbf{Abstract}\par\smallskip
\begin{minipage}{\dimexpr\paperwidth-9.9cm}
We present an algorithm that computes Friedl and Lück's \emph{twisted $L^2$-Euler characteristic} for a suitable CW complex, employing Oki's \emph{matrix expansion algorithm} to indirectly evaluate the Dieudonné determinant. The algorithm needs to run for an extremely long time to certify its outputs, but a truncated, human-assisted version produces very good results in many cases, such as hyperbolic link complements, closed census $3$-manifolds, free-by-cyclic groups, and higher-dimensional examples, such as the fiber of the Ratcliffe-Tschantz manifold.
\end{minipage}
\par\endgroup

\section*{Introduction}

The twisted $L^2$-Euler characteristic is an $L^2$-invariant for a CW complex $X$ introduced by Friedl and Lück in~\cite{l2thur}. It consists of a homogeneous integer-valued function on the first cohomology $H^1(X; \mathbb Z)$ and equals, up to sign, the celebrated \emph{Thurston norm} on a large class of $3$-manifolds, while also exposing the fiber's Euler characteristic if $X$ fibers over $S^1$. For this reason, it is a natural candidate for a Thurston norm-like invariant in the study of odd-dimensional manifolds.

This invariant is strongly related to the \emph{von Neumann algebra} of the fundamental group $G$ of $X$, a noncommutative ring which generally lacks an explicit, finite description for its elements. By assuming the Atiyah conjecture for $G$, we can recast the twisted $L^2$-Euler characteristic as a homology invariant of a certain complex over the \emph{Linnell skew field} of $G$. Even so, this object proves to be quite difficult to handle: the theory of modules over skew fields, while very similar to the usual commutative linear algebra, replaces the determinant of a matrix with the computationally elusive \emph{Dieudonné determinant}.

In this paper, we apply Oki's \emph{matrix expansion algorithm}~\cite{oki} to approximate different valuations of this determinant without computing it explicitly. This ultimately translates to computing twisted $L^2$-Euler characteristic at different first cohomology classes, leading to our main result:

\begin{thm*}
    There exists an algorithm that, given a finite $L^2$-acyclic CW complex $M$, such that its fundamental group $G$ is residually finite and satisfies the Atiyah conjecture, and a character $\phi: G \to \ZZ$, computes the twisted $L^2$-Euler characteristic $\chi^{(2)}(\til M; \phi)$.
\end{thm*}

In general, it would seem that this method does not determine the entire invariant in a finite amount of evaluations. In fact, assuming the Atiyah conjecture, the twisted $L^2$-Euler characteristic $\chi^{(2)}(\til M; {-})$ is a difference of two seminorms having polytopal unit balls. However, if $\chi^{(2)}(\til M; {-})$ is (plus or minus) a single seminorm, as is the case for many $3$-manifolds, then any enumeration of integral cohomology classes will eventually subsume all vertices of the corresponding polytope, plus one interior point for each facet. The values of the seminorm at these points determine the polytope, giving us a stopping condition that can be checked algorithmically. Hence, running the algorithm on all cohomology classes gives us:

\begin{cor*}
    In the same context as the preceding theorem, there exists an algorithm that computes $\chi^{(2)}(\til M; {-})$ in finite time whenever it is a seminorm up to sign.
\end{cor*}

The approximation in our application of Oki's algorithm comes from Lück's approximation theorem and is not \emph{a priori} effectively computable; as such, we discuss these matters in light of Löh and Uschold's recent developments~\cite{l2comp}, which give effective bounds for its convergence. While this suffices to make our procedure into an exact algorithm, the bounds are extremely loose and impractical with respect to implementation.

However, tests show that a truncated, human-assisted version of the algorithm produces very good experimental results in a wide variety of cases, such as the following.
\begin{itemize}
    \item \textbf{Dunfield's link $\mathrm{L10n14}$.} As mentioned at the beginning, in the case of $3$-manifolds, our algorithm provides an experimental method for the computation of the Thurston norm; we remark that exact algorithms exist, both for manifolds having $2$-generator $1$-relator fundamental groups~\cite{thurst-fox} and for general manifolds, via normal surface methods~\cite{thurst-normal}. In order to get a quick estimate, one can also bound the Thurston norm with the Alexander norm, by \emph{McMullen's inequality}~\cite{mcmullen}; among other link complements, we demonstrate an example, first constructed by Dunfield~\cite{dunfield} as the complement of the two-component link L10n14, where our algorithm can correctly approximate the former even when the two norms differ.
    \item \textbf{Closed census $3$-manifolds.} We apply our algorithm to all closed manifolds in the census with rank-$1$ first homology. For these examples, the Thurston norm is naturally reduced to a single integer, greatly simplifying experimentation. Surprisingly, our algorithm produces the exact Thurston norm, even when Lück's theorem is applied with the trivial quotient. We also investigate the unique census manifold with rank-$2$ homology $\mathrm{v1539(5,1)}$, with good results.
    \item \textbf{Free-by-cyclic groups.} Free-by-cyclic groups can be thought of as extending the class of $3$-manifolds with boundary, by analogy with the semidirect product structure of such a manifold when it fibers over the circle. Considering this, we study a couple of randomly generated examples of free-by-cyclic groups.
    \item \textbf{The fiber of the Ratcliffe-Tschantz manifold.} Lastly, of particular interest is the computation of the $L^2$-Betti numbers of the fiber $F$ of the Ratcliffe-Tschantz $5$-manifold studied in~\cite{imm}, corroborating, in this case, the \emph{Singer conjecture} about the vanishing of the non-middle-dimensional $L^2$-Betti numbers.
    In this $4$-dimensional example, we skip the matrix expansion step, computing untwisted $L^2$-Betti numbers, which equal the twisted $L^2$-Betti numbers of the whole Ratcliffe-Tschantz manifold.
\end{itemize}
\subsection*{Structure of the paper}
In Section~1 we give an overview of the concepts underlying the theory of $L^2$-invariants, leading to the definition of the $L^2$-Betti numbers and the twisted $L^2$-Euler characteristic in Section~2. Section~3 introduces a way to view elements of the Linnell skew field $\linn(G)$ as formal differences of polytopes with integral vertices. Then, in Section~4, we define the universal $L^2$-torsion, an invariant whose Dieudonné determinant corresponds to a polytope whose thickness function is the twisted $L^2$-Euler characteristic. 
In Section~5, we describe our algorithm, introducing Oki's matrix expansion algorithm and Lück's approximation theorem; we discuss implementation details in Section~6.
In Section~7, we analyze the results of various experiments.
Finally, in Section~8, we present some final remarks and possibilities for future research.

\subsection*{Acknowledgments}
I would like to thank Dawid Kielak for suggesting the topic of $L^2$-Euler characteristics and for an interesting discussion, my advisor Bruno Martelli for his support during the writing of this paper, and Matthias Uschold for a conversation which led me to write Remark~\ref{rmk:selfadjoint}. I am also grateful to an anonymous referee for many helpful comments and suggestions.

\clearpage
\renewcommand{\baselinestretch}{0.93}\normalsize
\tableofcontents
\renewcommand{\baselinestretch}{1.0}\normalsize
\clearpage

\section{Preliminaries on \texorpdfstring{$L^2$-invariants}{L\^2-invariants}}
In this section, we review some basic notions in the theory of $L^2$-invariants. In what follows we take $G$ to be any group, even if in our applications we will always consider countable, torsion-free groups. Rings are unital but not necessarily commutative, and modules are left modules unless explicitly stated.

\subsection{The von Neumann algebra \texorpdfstring{$\mathcal N(G)$}{N(G)}}
Considering $G$ as a discrete measure space, we can form the complex Hilbert space $\ell^2(G)$ of square-summable sequences indexed by $G$, where the inner product
\begin{equation}
    \big\langle \sum_{g \in G} \alpha_g g, \sum_{g \in G} \beta_g g \big\rangle \coloneqq \sum_{g \in G} \alpha_g \overline{\beta_g}
\end{equation}
is antilinear in the second argument. There is a natural left action $G \curvearrowright \ell^2(G)$, defined by formal multiplication; this leads to the following definition.

\begin{defin}
The \emph{von Neumann algebra} $\mathcal N(G)$ is the algebra of $G$-equivariant endomorphisms of $\ell^2(G)$, that is, all linear bounded operators $\ell^2(G) \to \ell^2(G)$ that commute with the left action by $G$.
\end{defin}

To get a better grasp on this object, we note that any element $f \in \mathcal N(G)$ is determined by the image of the identity element $e \in G \subset \ell^2(G)$:
\begin{equation}
    f\left( \sum_{g\in G} \alpha_g g \right) = \sum_{g\in G} \alpha_g f(g)
        = \sum_{g\in G} \alpha_g g f(e).
\end{equation}
Therefore, we can identify $\mathcal N(G)$ with a subset of $\ell^2(G)$. Moreover, if we write $f(e) = \sum_{h \in G} \varphi_h h$, we have
\begin{equation}
    \sum_{g\in G} \alpha_g g f(e) = \sum_{g\in G} \alpha_g \sum_{h \in G} \varphi_h gh,
\end{equation}
which is formally an action by \emph{right} multiplication.

The complex number $\varphi_e = \langle f(e), e \rangle$ is called the \emph{von Neumann trace}, $\tr_{\mathcal N(G)}(f)$ and will play a fundamental role in defining the \emph{von Neumann dimension}.

\subsection{Von Neumann dimension for \texorpdfstring{$\mathcal N(G)$-modules}{N(G)-modules}}
The theory of modules over certain rings, such as fields, skew fields and PIDs, is characterized by the existence of a well-behaved \emph{dimension} (or \emph{rank}) function. This concept also applies to modules over $\mathcal N(G)$, with the peculiarity that the dimension is not necessarily an integer.

We start with some general definitions.
\begin{defin}
Let $R$ be a ring and let $M\subseteq N$ be two $R$-modules. The \emph{closure} of $M$ in $N$ is
\[
    \overline{M} \coloneqq \bigcap\;\{\ker f \mid f \in \Hom_R(N, R), \restr{f}{M} \equiv 0\}.
\]
\end{defin}

\begin{defin}
A \emph{predimension} on $R$ is a function\footnote{Strictly speaking, the term ``function'' should be taken to mean ``class function'', to avoid foundational issues when the domain is a proper class.}
\[ \pdim\colon \{\text{finitely generated projective $R$-modules}\} \to [0,+\infty) \]
such that:
\begin{enumerate}[label=(\arabic*)]
    \item $P \simeq Q \implies \pdim(P) = \pdim(Q)$;
    \item $\pdim(P\oplus Q) = \pdim(P) + \pdim(Q)$;
    \item if $K$ is a submodule of $Q$, then its closure $\overline K$ is a direct summand of $Q$ and
        \[
            \pdim{\overline K} = \sup\;\{ \pdim(P) \mid P \subseteq K \text{ fin.\,gen.\,projective}\}.
        \]
\end{enumerate}
\end{defin}

Such a function can be defined on $\mathcal N(G)$ as follows. Let $P$ be a finitely generated projective $\mathcal N(G)$-module. Then, $P$ is a quotient of some finitely generated free module $F = \mathcal N(G)^n$; it is also a direct summand by projectivity. The projection on $P$ is an endomorphism of $N(G)^n$ and can be seen as a square matrix $A$ with coefficients in $\mathcal N(G)$. Finally, we define
\begin{equation}
    \pdim_{\mathcal N(G)}(P) \coloneqq \tr_{\mathcal N(G)}(A) \coloneqq \sum_{i=1}^n \tr_{\mathcal N(G)}(A_{ii}).
\end{equation}

Extending this definition to a genuine dimension function can also be done in a general framework, as follows.

\begin{defin}
Given a predimension $\pdim$ on $R$, we can define a dimension function on all $R$-modules:
\[
    \dim(M) \coloneqq \sup\;\{ \pdim(P) \mid P \subseteq M \text{ fin.\,gen.\,projective}\} \in [0, +\infty].
\]
\end{defin}

\begin{thm}[{\cite[Theorem~6.7]{l2}}]
Let $\pdim$ be a predimension on $R$. Then:
\begin{enumerate}[label=(\arabic*)]
    \item $\dim$ agrees with $\pdim$ on the class of finitely generated projective $R$-modules;
    \item given a short exact sequence of $R$-modules
    \[
        \begin{tikzcd}
        0 \arrow[r] & A \arrow[r] & B \arrow[r] & C \arrow[r] & 0,
        \end{tikzcd}
    \]
    we have $\dim(B) = \dim(A) + \dim(C)$, where the sum is extended to $[0,+\infty]$ in the natural way;
    \item if $\{M_i\mid i \in I\}$ is a family of submodules of $M$ such that any two $M_i$, $M_j$ are contained in some $M_k$, and $\bigcup_{i \in I} M_i = M$, then
    \[
        \dim(M) = \sup\;\{\dim(M_i)\mid i \in I\};
    \]
    \item if $K \subseteq M$ is a submodule, then $\dim(\overline K) = \dim K$;
    \item properties (1)--(4) uniquely characterize $\dim$.
\end{enumerate}
\end{thm}

It follows, of course, that all $\mathcal N(G)$-modules have a well-defined dimension $\dim_{\mathcal N(G)}$.

\section{\texorpdfstring{$L^2$-Betti}{L\^2-Betti} numbers}
If $X$ is a $G$-CW complex, i.e.\ a CW complex with a properly discontinuous, cellular $G$-action that is free on open cells, then its cellular chain groups $C_n(X)$ have a natural left $\mathbb ZG$-module structure, which is respected by the boundary maps. On the other hand, the von Neumann algebra contains $\mathbb ZG$ as a subring and is a right module over it. Therefore, we can form the tensor product of $\mathcal N(G)$ and $C_\blt(X)$ over $\mathbb ZG$, and use it to define the \emph{$L^2$-Betti numbers}.

\begin{defin}
    The \emph{$L^2$-Betti numbers} of a $G$-CW complex $X$ are defined as
    \[
        b^{(2)}_n(X; G) \coloneqq \dim_{\mathcal N(G)} H_n(\mathcal N(G) \otimes_{\mathbb ZG} C_\blt(X)).
    \]
    It is also natural to define the \emph{$L^2$-Euler characteristic}
    \[
        \chi^{(2)}(X; \neu (G)) \coloneqq \sum_{n \ge 0} (-1)^n \cdot b^{(2)}_n(X; \neu (G))
    \]
    whenever the series is absolutely convergent. We shall simply write $b^{(2)}_n(X)$ and $\chi^{(2)}(X)$ if the choice of $G$ is clear from the context.
    \label{def:l2betticw}
\end{defin}
The prime example of a $G$-CW complex is the universal covering $\widetilde Y$ of a connected CW complex $Y$, where of course $G = \pi_1(Y)$. Here the $C_n(\widetilde Y)$ are even free over $\mathbb ZG$, generated by $G$-orbits of $n$-cells in $\widetilde Y$, that is, inverse images of $n$-cells of $Y$.

By inspecting the definition, we note that a $G$-homotopy of $G$-complexes induces a $\ZZ G$-chain homotopy between the cellular chain complexes, which becomes a $\mathcal N(G)$-chain homotopy after tensoring. Hence, the $L^2$-Betti numbers are $G$-homotopy invariants (see~\cite[Theorem~6.54~(1b)]{l2} for a more general statement), so $b^{(2)}_n(\til Y)$ and $\chi^{(2)}(\til Y)$ are homotopy invariants of $Y$.

Somewhat surprisingly, for a finite CW complex $Y$ we have $\chi^{(2)}(\widetilde Y) = \chi(Y)$~\cite[Theorem~1.35, (2)]{l2}.

\begin{remark}
Definition~\ref{def:l2betticw} actually applies to the extremely general case of a chain complex $C_\blt$ of $\mathbb ZG$-modules: in fact, we can define its $L^2$-Betti numbers simply as
\[
    b^{(2)}_n(C_\blt; \mathcal N(G)) \coloneqq \dim_{\mathcal N(G)} H_n(\mathcal N(G) \otimes_{\mathbb ZG} C_\blt).
\]
\end{remark}

Unlike classical Betti numbers, their $L^2$ counterparts enjoy a multiplicative property when passing to finite coverings or, equivalently, subgroups of $G$ of finite index:
\begin{prop}[compare~{\cite[Theorem~6.54~(6)]{l2}}]
Let $X$ be a $G$-CW complex and let $i\colon H\hookrightarrow G$ be the inclusion of a subgroup of finite index. Then $X$ is also an $H$-CW complex $i^*X$ and
\[
    b_n^{(2)}(i^*X; \mathcal N(H)) = [G:H] \cdot  b_n^{(2)}(X; \mathcal N(G)).
\]
As a consequence, if $Y$ is a CW-complex and $Z$ is a finite covering of degree d, then
\[
    b_n^{(2)}(\widetilde Z) = d\cdot b_n^{(2)}(\widetilde Y).
\]
\end{prop}

A special class of CW complexes is characterized by the vanishing of all $L^2$-Betti numbers. Such complexes are called \emph{$L^2$-acyclic} and include all hyperbolic odd-dimensional closed manifolds and all mapping tori of finite connected CW-complexes (such as fibrations over the circle). In this case, we can obtain finer invariants by twisting with an infinite-dimensional $\mathbb ZG$-module.

\begin{defin}
Let $X$ be a $G$-CW complex and let $\phi\colon G \to \mathbb Z$ be a character. Then $G$ acts on the Laurent polynomial ring $\mathbb Z[t, t^{-1}]$ via $g \cdot p \coloneqq t^{\phi(g)}p$, and we can define the $\mathbb ZG$-chain complex
\[
    \overline C_\blt(X) \coloneqq C_\blt(X) \otimes_\mathbb Z \mathbb Z[t, t^{-1}]
\]
with the diagonal action $g\cdot (\sigma \otimes p) \coloneqq g\cdot \sigma \otimes g\cdot p$.

The \emph{twisted $L^2$-Betti numbers} of $X$ are given by
\[
    b_n^{(2)}(X; \mathcal N(G), \phi) \coloneqq \dim_{\mathcal N(G)} H_n(\mathcal N(G) \otimes_{\mathbb ZG} \overline C_\blt(X)).
\]
\end{defin}
\begin{defin}
    A $G$-CW complex $X$ is said to be \emph{$\phi$-$L^2$-finite} if
    \[
        \sum_{n \ge 0} b_n^{(2)}(X; \mathcal N(G), \phi) < +\infty.
    \]
    If this is the case, we call
    \[
        \chi^{(2)}(X; \mathcal N(G), \phi) \coloneqq \sum_{n \ge 0} (-1)^n\cdot b_n^{(2)}(X; \mathcal N(G), \phi)
    \]
    the \emph{$\phi$-twisted $L^2$-Euler characteristic} of $X$.
\end{defin}
These invariants also have desirable properties, as evidenced by the following results of Friedl-Lück:
\begin{prop}[{\cite[Theorem~2.5, Lemma~2.6]{l2thur}}]%
\label{prop:twisted-l2}
Let $X$ be a $G$-CW complex and let $\phi\colon G \to \mathbb Z$. Then:
\begin{enumerate}[label=(\arabic*)]
    \item Let $i\colon H \to G$ be the inclusion of a subgroup of finite index; then the $H$-CW complex $i^* X$ is $(\phi \circ i)$-$L^2$-finite if and only if $X$ is $\phi$-$L^2$-finite, and if this is the case, then
    \[
        \chi^{(2)}(i^* X; \mathcal N(H), \phi\circ i)
        = [G:H]\cdot\chi^{(2)}(X; \mathcal N(G), \phi);
    \]
    \item For every integer $k \ge 1$, $X$ is $\phi$-$L^2$-finite if and only if it is $(k\phi)$-$L^2$-finite, and if this is the case, then
\[\chi^{(2)}(X; \mathcal N(G), k\phi) = k \cdot \chi^{(2)}(X; \mathcal N(G), \phi);\]
    \item Suppose that $\phi$ is the trivial character; then X is $\phi$-$L^2$-finite if and only if we have $b^{(2)}_n (X; \mathcal N(G)) = 0$ for all $n \ge 0$. If this is the case, then \[\chi^{(2)}(X; \mathcal N(G), 0) = 0.\]
    \item If $\phi$ is primitive, i.e.\;surjective, and $i\colon K \to G$ is the inclusion of its kernel, then $X$ is $\phi$-$L^2$-finite if and only if
    $b^{(2)}_n(i^*X; \mathcal N(K)) < \infty$ for all $n \in \mathbb N$.
If this is the case, then
\[b^{(2)}_n(X; \mathcal N(G), \phi) = b^{(2)}_n(i^*X; \mathcal N(K)).\]
\end{enumerate}
\end{prop}

This invariant is very valuable from a geometric perspective. Just like the untwisted case, Proposition~\ref{prop:twisted-l2}~(1) gives multiplicativity with respect to finite coverings; however, (4) gives us control on \emph{infinite cyclic coverings}. If $Y$ is a finite connected CW complex and $X$ is its universal cover $G$-CW complex, then the $i^*X$ appearing in the statement is the $K$-CW complex associated to $\overline Y_\infty$, the $\ZZ$-covering of $Y$ defined by $\phi$. Whenever $\phi$ is \emph{fibered}, that is, induced by a fibration $Y \surj S^1$ via the $\pi_1$ functor, this covering deformation retracts onto the fiber $F$. Therefore, we have
\begin{equation}
    \chi^{(2)}(\til Y; \phi) = \chi^{(2)}(\til{\overline{Y}}_\infty) = \chi^{(2)}(\til{F}) = \chi(F).
\end{equation}

This is reminiscent of the \emph{Thurston norm}, which equals the negative of the fiber Euler characteristic for fibered classes of a $3$-manifold. Indeed, we have:

\begin{prop}[{\cite[Theorem~0.2]{l2thur}}]\label{prop:thurston-l2}
Let $M\ne S^2 \times D^1$ be a compact, connected, orientable, irreducible $3$-manifold, whose boundary is a union of zero or more tori and whose fundamental group is infinite. Then, for any $\phi \in H^1(M; \ZZ)$,
\[-\chi^{(2)}(\til M; \phi) = x_M (\phi),\]
where $x_M$ is the Thurston seminorm of $M$.
\end{prop}

Hence, the twisted $L^2$-Euler characteristic is a natural extension of the Thurston norm to spaces that are more general than $3$-manifolds.

The condition of $\phi$-$L^2$-finiteness, necessary for its definition, may seem daunting at first, but it can be elegantly dealt with by assuming the so-called \emph{Atiyah conjecture} for the fundamental group $G$. This will also ensure that the twisted $L^2$-Euler characteristic is especially well behaved.

\subsection{The Atiyah conjecture}\label{sec:atiyah}
Let $G$ be a torsion-free group; examples include, notably, all hyperbolic manifold groups.
\begin{defin}[Atiyah conjecture]
We say that $G$ satisfies the \emph{Atiyah conjecture} if, for any matrix $A \in \mathrm{M}(m, n, \mathbb QG)$, acting by right multiplication as a map \[r_A\colon (\mathcal N(G))^m \to (\mathcal N(G))^n,\] the von Neumann dimension $\dim_{\mathcal N(G)} \ker(r_A)$ is an integer.
\end{defin}

This conjecture holds for a rather large class of torsion-free groups, containing all torsion-free fundamental groups of $3$-manifolds~\cite{atiyah-3mfd} and closed under taking subgroups~\cite[Theorem~3.2~(1)]{l2thur}. At the time of writing this paper, there are no known counterexamples to the Atiyah conjecture among torsion-free groups; otherwise, see~\cite{atiyah-torsion}.

This assumption opens a road to the computation of $L^2$-invariants by replacing the von Neumann algebra with the slightly less elusive \emph{Linnell skew field}.
\begin{defin}
Since $\mathcal N(G)$ satisfies the Ore condition with respect to its set of non-zero-divisors~\cite[Theorem~8.22~(1)]{l2}, we can define its Ore localization $\mathcal U(G)$. 
The \emph{Linnell skew field} $\mathcal D(G)$ is defined as the smallest subring of $\mathcal U(G)$ that contains $\mathbb QG$ and is division closed, i.e.\;it contains inverses of all the $\mathcal U(G)$-units in it.
\end{defin}

The following result by Friedl-Lück expands on the ideas we introduced earlier.

\begin{thm}[{\cite[Theorem~3.8]{l2thur}}]\label{thm:DG}
    Let $G$ be a torsion-free group that satisfies the Atiyah conjecture and let $C_\blt$ be a chain complex of projective $\mathbb QG$-modules. Then:
    \begin{enumerate}[label=(\arabic*)]
        \item The ring $\mathcal D(G)$ is a skew field.
        \item For all $n\ge 0$,
        \[
            b_n^{(2)}(C_\blt; \mathcal N(G))
                = \dim_{\mathcal D(G)} H_n(\mathcal D(G) \otimes_{\mathbb QG} C_\blt),
        \]
        which is either a natural number or $+\infty$.
        \item Let $\phi\colon G \to \mathbb Z$ be surjective with kernel $K$. Then $G$ can be identified with the semidirect product $K \rtimes_{c} \mathbb Z$, where the $\mathbb Z$-factor is generated by any element $u$ such that $\phi(u) = 1$ acting on $K$ via the conjugation $c$. The character $\phi$ becomes the canonical projection $G = K \rtimes_c \mathbb Z \to \mathbb Z$.
        
        The conjugation induces an automorphism $t \in \mathrm{Aut}(\mathcal D(K))$, enabling us to define the \emph{skew Laurent polynomial ring} $\mathcal D(K)_t[u^{\pm 1}]$, with multiplication defined on monomials as $au^i \cdot bu^j \coloneqq at^{-i}(b)u^{i+j}$.

        Then, this ring is a (noncommutative) principal ideal domain and it satisfies the Ore condition with respect to its subset $T$ of non-zero elements. Moreover, there is a canonical isomorphism of skew fields
        \[
            T^{-1} \mathcal D(K)_t[u^{\pm 1}] \xlongrightarrow{\sim} \mathcal D(G).
        \]

        \item Let $\phi\colon G \to \mathbb Z$ be surjective with kernel $K$ and let $i: K \hookrightarrow G$ be the inclusion. Suppose that $C_\blt$ is finitely generated and $b^{(2)}_n(\mathcal N(G) \otimes_{\mathbb QG} C_\blt) = 0$ for all $n \ge 0$.
        
        Consider the $\mathbb QK$-chain complex $i^* C_\blt$ obtained by restriction. Then, for all $n \ge 0$:
        \begin{itemize}
            \item the modules $H_n(\mathcal D(K) \otimes_{\mathbb QK} i^*C_\blt)$ and $H_n(\mathcal D(K)_t[u^{\pm 1}] \otimes_{\mathbb QG} C_\blt)$ are finitely generated free over $\mathcal D(K)$;
            \item the $L^2$-Betti number $b^{(2)}_n (i^*C_\blt; \mathcal N(K))$ is finite and, in particular,
            \begin{align*}
                b^{(2)}_n (i^*C_\blt; \mathcal N(K)) &= \dim_{\mathcal D(K)} (H_n(\mathcal D(K)_t[u^{\pm 1}] \otimes_{\mathbb QG} C_\blt)) \\
                &= \dim_{\mathcal D(K)} (H_n(\mathcal D(K) \otimes_{\mathbb QK} i^*C_\blt)).
            \end{align*}
        \end{itemize} 
    \end{enumerate}
\end{thm}

As anticipated, this theorem simplifies the treatment of the twisted $L^2$-Euler characteristic. Indeed, if $X$ is an $L^2$-acyclic finite connected CW complex, $G \coloneqq \pi_1(X)$ satisfies the Atiyah conjecture and $\phi$ is a primitive character, then all the $\phi$-twisted $L^2$-Betti numbers are integers, being dimensions of vector spaces over skew fields. By homogeneity, this also holds for non-primitive $\phi$; therefore, the twisted $L^2$-Euler characteristic $\chi^{(2)}(\til X; \mathcal N(G), \phi)$ is always an integer.

For this reason, we will assume that the Atiyah conjecture is satisfied for all the groups we consider in the remainder of this article.

\section{The polytope homomorphism}\label{sec:polyhom}
Having translated all our machinery in the language of skew fields, we will see how to unify all the twisted $L^2$-Euler characteristics for every character $\phi$ into a single object, by leveraging the \emph{Dieudonné determinant} for skew field square matrices.

Consider a finitely generated torsion-free group $G$ with a free abelian quotient $\nu\colon G \surj H$.

\begin{defin}
We call \emph{integral polytopes} in $H \otimes \RR$ the subsets obtained by taking convex hulls of finitely many points of $H$. Such polytopes form a commutative, cancellative monoid $\mathfrak P(H)$ under the Minkowski sum operation; we define the \emph{integral polytope group} $\mathcal P(H)$ to be the Grothendieck group of $\mathfrak P(H)$.
\end{defin} 

Elements of $\mathcal P(H)$ can be seen as formal differences of two polytopes, where $P-Q = P'-Q'$ if and only if $P+Q' = P'+Q$ in $\mathfrak P(H)$. Moreover, since $\mathfrak P(H)$ is cancellative, it is actually embedded in its Grothendieck group $\mathcal P(H)$ as the subset consisting of those differences that can be written in the form $P-0$.

The next step is a multidimensional generalization of Theorem~\ref{thm:DG}~(3), 
obtained with $\nu$ in place of $\phi$:

\begin{prop}[{\cite[(6.1)]{l2thur}}]
Let $\nu\colon G \surj H$ be a free abelian quotient as before, and let $K$ be its kernel. Then the crossed product ring $\linn(K) * H$ admits an Ore localization with respect to the set $T$ of its non-zero elements, and there is an isomorphism
\[
    T^{-1} (\linn(K) * H) \simeq \linn(G)
\]
of $\linn (K)$-modules.
\end{prop}

\begin{remark}
For our purposes, the above \emph{crossed product} can be thought of as a \emph{multivariate skew polynomial ring} $\linn(K)[u_1^{\pm 1}, \dots, u_r^{\pm 1}]$, where $\{u_1, \dots, u_r\}\subset G$ is a (fixed) lift of a basis for $H$. Elements are sums of monomials of shape $ku_1^{\alpha_1} \dots u_r^{\alpha_r}$,
or $k\midx{u}^\midx{\alpha}$ (with $\alpha \in H$) using multi-index notation. Multiplication is defined in a natural way as
\begin{equation}
    k\midx{u}^\midx{\alpha} \cdot h\midx{u}^\midx{\beta} \defeq
    k(\midx{u}^\midx{\alpha}h\midx{u}^{-\midx{\alpha}})
    \midx{u}^{\midx{\alpha}+\midx{\beta}}.
\end{equation}
Just like in Theorem~\ref{thm:DG}~(3), since $\midx{u}^\midx{\alpha}$ is an element of $G$, it induces an automorphism of $\linn(K)$ by conjugation, making the above definition sound. For a formal, general definition of crossed product, see e.g.~\cite[Section~10.3.2]{l2}.
\end{remark}

Owing to this, we will make extensive use of a polynomial-like notation \begin{equation}
    q(u_1,\dots,u_r)^{-1}p(u_1,\dots,u_r)
\end{equation}
for elements of the Linnell skew field $\linn(G)$.

Every non-zero polynomial $p \in \linn(K)[u_1^{\pm 1}, \dots, u_r^{\pm 1}]$ has a \emph{noncommutative Newton polytope} $P(p)$, given by the convex hull in $H \otimes \RR$ of all exponents $\alpha$ of its non-zero monomials. It can be proved~\cite[Section~6.1]{l2thur} that $P$ is independent of the choice of lifts $u_1, \dots, u_n$ and sends products to Minkowski sums, making it a monoid homomorphism
\begin{equation}
    P\colon (\linn(K)[u_1^{\pm 1}, \dots, u_r^{\pm 1}]\setminus \{0\}, \cdot)
        \to \polym(H).
\end{equation}
By localizing, this can be augmented to a group homomorphism
\begin{equation}
    P_\nu\colon \linn(G)^\times_\mathrm{ab} \to \poly(H),\heqsep P_\nu(q^{-1}p) \defeq P(p)-P(q),
\end{equation}
where we can abelianize the domain since the codomain is abelian.

Consider now a primitive character $\phi$ containing $K$ in its kernel; clearly, $\phi$ factors as $\omega \circ \nu$ for some $\omega \colon H \surj \ZZ$. Any $p \in \linn(K)[u_1^{\pm 1}, \dots, u_r^{\pm 1}]$ can, of course, be seen as a single-variable skew Laurent polynomial in $\linn(\ker \phi)[v^{\pm 1}]$ whose \emph{degree} (the difference between the greatest and the least exponents of $v$ in $p$) can be directly read off from $P(p)$.
In fact, every multivariate monomial $k\midx{u}^\midx{\alpha}$ contributes to the monomial $v^{\omega(\alpha)}$, and there can be no simplifications between terms with the same $\omega(\alpha)$, as they are still linearly independent over $\linn(K)$.

Therefore, the degree $\deg_\phi p$ of $p$, as a polynomial over $\linn(\ker \phi)$, is simply the \emph{thickness} of $P(p)$ along the $\omega$ direction. Since $P(p)$ is convex and compact, its thickness function is an integer-valued \emph{seminorm} on $\Hom(H, \ZZ)$, which can be extended by homogeneity and continuity to a full-fledged seminorm on $\Hom(H, \RR)$.

We can also define the degree of a rational element $q^{-1}p \in \linn(G)^\times$ as $\deg_\phi p - \deg_\phi q$. With an analogous reasoning, we get that the degree is the \emph{difference} of the two seminorms associated to $P(p)$ and $P(q)$.

\begin{remark}
    As the degree is a homomorphism to $\ZZ$, it is in fact well defined for elements of $\linn(G)^\times_\mathrm{ab}$.
\end{remark}
\begin{remark}
    The degree is independent of the quotient $\nu$. In fact, it is easy to see that taking a smaller quotient corresponds to restricting the seminorm to a subspace or to projecting the polytope onto a subspace. Therefore, we will generally consider $\nu = \ab \colon G \surj \ab(G)$, the largest free abelian quotient.
\end{remark}

\subsection{The Dieudonné determinant}\label{sec:dieudonne}
The determinant is a natural invariant of square matrices over commutative rings. Apart from its well-known desirable properties, such as multilinearity and multiplicativity, it can also capture some fine-grained information about endomorphisms over PIDs. As an example:
\begin{prop} \label{prop:degdet-c}
We have the following:
    \begin{itemize}
        \item[--] if $A\in \mathrm{M}(n, \ZZ)$ is invertible over $\QQ$, then $\coker A$ has cardinality $|{\det A}|$;
        \item[--] if $A\in \mathrm{M}(n, \mathbb K[x])$ is invertible over $\mathbb K(x)$, then $\coker A$ has $\mathbb K$-dimension $\deg {\det A}$.
    \end{itemize}
\end{prop}
Both facts follow easily from the Smith normal form and can be generalized extensively.

In the context of skew fields, noncommutativity makes it impossible to define a determinant with the sum-of-permutations approach, as multiplicativity would fail. A better definition starts with the \emph{Bruhat decomposition} of a matrix:
\begin{prop}[Bruhat decomposition,~{\cite[Theorem~9.2.2]{cohn}}]
    Let $A$ be an invertible square matrix over a skew field $F$. Then there is a decomposition $A = LDPU$, where $L$ and $U$ are lower and upper unitriangular, $D$ is diagonal and $P$ is a permutation matrix. Moreover, $D$ and $P$ are uniquely determined.
\end{prop}
Now we can define the \emph{Dieudonné determinant}
\begin{equation}
    \det A \defeq \hspace{0.1em}\operatorname{sign}(P)\cdot d_1 \dots d_n \!\!\mod [F^\times,F^\times] \in F^\times_\mathrm{ab},
\end{equation}
where $\operatorname{sign}(P)$ is the sign of the corresponding permutation, and $D = \operatorname{diag}(d_1,\dots,d_n)$.

The Dieudonné determinant is the unique map $\det \colon \mathrm{GL}(n,F) \to F^\times_\mathrm{ab}$ such that:
\begin{enumerate}[label=(\arabic*)]
    \item $\det AB = \det A \det B$ for $A, B\in \mathrm{GL}(n,F)$;
    \item if $E = I_n + cE_{ij}$ with $i \ne j$ is an elementary matrix, then $\det E = 1$;
    \item $\det \operatorname{diag}(d_1,\dots,d_n) = d_1 \dots d_n \!\!\mod [F^\times,F^\times]$.
\end{enumerate}
Of course, these properties (see~\cite[Section~3.1]{oki}) are enabled by descending to the coarser group $F^\times_\mathrm{ab}$. 

Let us now retrace our steps by asking if there is a noncommutative version of Proposition~\ref{prop:degdet-c}. Indeed, we can consider the determinant of a suitable matrix $A$, apply the degree function we defined earlier and relate the result to the cokernel of $A$.

\begin{prop} \label{prop:degdet-nc}
    Given a free abelian quotient $\nu \colon G \surj H$ with kernel $K$, let $A$ be a square matrix with coefficients in $\linn(K)[u_1^{\pm 1}, \dots, u_r^{\pm 1}]$, invertible over $\linn(G)$. Let $\phi = \omega \circ \nu$ be a primitive character with kernel $K_\phi$. Then
    \[
        \dim_{\linn(K_\phi)} \coker A = \deg_\phi \det A.
    \]
\end{prop}
\begin{proof}
This is essentially proved in~\cite[Lemma~6.16]{l2thur}. The main idea is to consider $A$ as a matrix over $\linn(K_\phi)[v^{\pm 1}]$ (a noncommutative PID), then bring it into the Smith normal form with elementary moves, just like in the commutative case.
\end{proof}
\begin{remark}
This argument also proves that $\det A$ has a representative in $\linn(K_\phi)[v^{\pm 1}]$: the Dieudonné determinant is the product of the diagonal entries in the Smith normal form, up to left and right multiplication by monomials that accumulate during the diagonalization process.
\end{remark}

Comparing Proposition~\ref{prop:degdet-nc} with Theorem~\ref{thm:DG}, we get a tentative first step toward the computation of the twisted $L^2$-Euler characteristic. Indeed, let $H = \ab(G) \simeq \ZZ^r$ and $K = \ker \ab$ and let $C_\blt$ be the following $\ZZ G$-chain complex:
\begin{equation}
\begin{tikzcd}
    \dots \arrow[r] & 0 \arrow[r] & C_1 \arrow[r, "A"] & C_0 \arrow[r] & 0
\end{tikzcd}
\end{equation}
with $A \in \mathrm{M}(n, \ZZ G)$ invertible over $\linn(G)$. If $\phi$ is a primitive character, we can consider the twisted $L^2$-Betti numbers
\begin{align}
b^{(2)}_0(C_\blt; \mathcal N(G), \phi) &=
    \dim_{\linn(K_\phi)} H_0(\linn(K_\phi)[v^{\pm}] \otimes_{\QQ G} C_\blt)
\\     &= \dim_{\linn(K_\phi)} \coker (\linn(K_\phi)[v^{\pm}] \otimes_{\QQ G} A)
\\ b^{(2)}_1(C_\blt; \mathcal N(G), \phi) &=
    \dim_{\linn(K_\phi)} H_1(\linn(K_\phi)[v^{\pm}] \otimes_{\QQ G} C_\blt)
\\     &= \dim_{\linn(K_\phi)} \ker (\linn(K_\phi)[v^{\pm}] \otimes_{\QQ G} A)
\end{align}
Since $A$ is injective over $\linn(G)$, it must be so over $\linn(K_\phi)[v^{\pm}]$, so $b^{(2)}_1(C_\blt; \mathcal N(G), \phi) = 0$. It follows that
\begin{equation}\label{eq:chi-el}
    \chi^{(2)}(C_\blt; \mathcal N(G), \phi) = b^{(2)}_0(C_\blt; \mathcal N(G), \phi) = \deg_\phi \det A.
\end{equation}

Therefore, in this toy example, we see that the twisted $L^2$-Euler characteristic is a rather concrete, combinatorial object, being the difference of two polyhedral seminorms; when extended to the real vector space $\Hom(G, \RR)$, it is also a continuous function. In the following section, we will see that this reasoning can be extended to \emph{all} $L^2$-acyclic free $\ZZ G$-chain complexes.

\section{Universal \texorpdfstring{$L^2$-torsion}{L\^2-torsion}}
Our aim for this section is to define some sort of matrix-valued invariant for $\ZZ G$-chain complexes in order to carry out the above plan; again, $G$ will be a finitely generated torsion-free group. We will mostly follow Friedl and Luck's article~\cite{uni-l2}.

\begin{defin}\label{def:weak-iso}
    An endomorphism $A \in \mathrm{M}(n, \ZZ G)$ is called a \emph{weak isomorphism} when \[\id_{\mathcal N(G)} \otimes_{\ZZ G} A \colon \ell^2(G)^n \to \ell^2(G)^n\] is an injective operator with dense image.
\end{defin}

\begin{remark}\label{rmk:weak-iso-DG}
    Since we are taking $G$ to satisfy the Atiyah conjecture, the property of being a weak isomorphism can be replaced by the more intuitive \emph{invertibility over $\linn(G)$}. This is ultimately a consequence of~\cite[Theorem~6.24]{l2} (see also~\cite[Lemma~1.21]{uni-l2}, where it is applied).
\end{remark}

\begin{defin}\label{def:weak-k1}
    The \emph{weak $K_1$ group} $K_1^w(\ZZ G)$ is the abelian group generated by all weak isomorphisms (with coefficients in $\ZZ G$), subject to the relations:
    \begin{enumerate}[label=(\arabic*)]
        \item $g \circ f = f + g$ for all $f,g$ compatible weak isomorphisms;
        \item $\begin{bmatrix} f & h \\ 0 & g \end{bmatrix} = f+g$
            for all $f,g$ weak isomorphisms and $h$ compatible $\ZZ G$-matrices.
    \end{enumerate}
\end{defin}

\begin{defin}\label{def:weak-whitehead}
    The \emph{weak Whitehead group} $\mathrm{Wh}^w(\ZZ G)$ is the quotient of $K_1^w(\ZZ G)$ by the subgroup generated by all the $1\times 1$ matrices $\begin{bmatrix} \pm g \end{bmatrix}$ for $g \in G$. We also define a finer quotient $\til{K}_1^w(\ZZ G) \defeq K_1^w(\ZZ G) / \langle -\id_{\ZZ G} \rangle$.
\end{defin}

\begin{defin}
    Let $C_\blt$ be a $\ZZ G$-chain complex with differentials $c_\blt$. A \emph{weak chain contraction} for $C_\blt$ is a pair $(u_\blt, \gamma_\blt)$, where:
    \begin{itemize}
        \item $u_n \colon C_n \to C_n$ is a weak isomorphism for all $n$;
        \item $\gamma_n \colon C_n \to C_{n+1}$ is a null homotopy, that is, $\gamma_n \circ u_n = u_{n+1}\circ \gamma_n$.
    \end{itemize}
\end{defin}

In the case of a free $\ZZ G$-chain complex $C_\blt$, the differentials $c_n$ can be thought of as $\ZZ G$-matrices acting by right multiplication, and have \emph{adjoints} $c_n^*$ induced by transposition and by the $\ZZ$-linear self-map of $\ZZ G$ sending $g$ to $g^{-1}$ (called the \emph{antipode} or \emph{coinverse} in the context of Hopf algebras). We can employ this construction to define the \emph{combinatorial Laplacian}, 
taking inspiration from the Laplacian of the de Rham complex:
\begin{equation}
    \Delta_n \coloneqq c_n^*c_n + c_{n+1}c_{n+1}^* \colon C_n \to C_n.
\end{equation}

It is easy to verify that $c_{\blt +1}$ is a null homotopy for $\Delta_\blt$. Indeed, since $c_\blt c_{\blt+1} = 0$, we have 
\begin{align}
    \Delta_n c_{n+1} &=
    (c_n^*c_n + c_{n+1}c_{n+1}^*)c_{n+1}
    \\ &= c_{n+1}c_{n+1}^*c_{n+1}
    \\ &= c_{n+1}(c_{n+1}^*c_{n+1} + c_{n+2}c_{n+2}^*)
    \\ &= c_{n+1}\Delta_{n+1}.
\end{align}
Hence, it is natural to ask if the combinatorial Laplacian is a weak isomorphism, making $(\Delta_\blt, c_{\blt+1})$ a weak chain contraction. 
By~\cite[Lemma~1.5]{uni-l2}, this is the case if and only if $C_\blt$ is $L^2$-acyclic. 

Now, let $C_\blt$ be a finitely generated free $L^2$-acyclic $\ZZ G$-chain complex such that $C_n \ne 0$ for finitely many $n$, and fix an unordered $\ZZ G$-basis for each $C_n$ (we will say that $C_\blt$ is \emph{finite based free $L^2$-acyclic}). If we define
\begin{equation}
    C_\mathrm{even} \defeq \bigoplus_{2\mid n} C_n, \heqsep C_\mathrm{odd} \defeq \bigoplus_{2\nmid n} C_n,
\end{equation}
the two modules inherit unordered bases, which are of the same cardinality by the $L^2$-acyclicity of $C_\blt$: let $b \colon C_\mathrm{even} \to C_\mathrm{odd}$ be an isomorphism induced by a bijection of the two bases.
If $f\colon C_\mathrm{odd} \to C_\mathrm{even}$ is a $\ZZ G$-homomorphism, then $[b\circ f]$ is a well-defined class in $\til{K}_1^w(\ZZ G)$, independent of the choice of bijection. Of course, it is necessary to quotient by the subgroup $\langle -\id_{\ZZ G} \rangle$ because the bases are unordered.

Now we can define the \emph{universal $L^2$-torsion} of $C_\blt$ as
\begin{equation}
    \rho_u^{(2)}(C_\blt) \defeq [(b\circ (\Delta c + c^*))] - [\Delta] \in \til{K}_1^w(\ZZ G),
\end{equation}
where $\Delta c + c^*$ and $\Delta$ are seen respectively as maps $C_\mathrm{odd} \to C_\mathrm{even}$ and $C_\mathrm{odd} \to C_\mathrm{odd}$.

On a more practical note, the interplay between even and odd degrees in the first term translates to large, unwieldy $\ZZ G$-matrices. Luckily, it is possible to express this invariant solely in terms of the Laplacians.

\begin{prop}[{\cite[Lemma~1.17]{uni-l2}}]%
\label{prop:altsum-lapl}
    Let $C_\blt$ be a finite based free $L^2$-acyclic $\ZZ G$-chain complex. Then its combinatorial Laplacians $\Delta_n$ are weak isomorphisms and
    \[
        \rho_u^{(2)}(C_\blt) + *(\rho_u^{(2)}(C_\blt))
        = -\sum_{n\ge 0} (-1)^n \cdot n \cdot [\Delta_n],
    \]
    where $* \colon \til{K}_1^w(\ZZ G) \to \til{K}_1^w(\ZZ G)$ is the involution induced by adjunction of $\ZZ G$-matrices.
\end{prop}

The importance of this invariant stems from a universal property that makes it, in a sense, the \emph{most general} invariant for the chain complexes of our interest.
\begin{defin}
    We say that a short exact sequence $0 \xrightarrow{} M_0 \xrightarrow{i} M_1 \xrightarrow{p} M_2 \xrightarrow{} 0$ of finite based free $\ZZ G$-modules, with bases $B_0, B_1, B_2$, is \emph{based exact} if $i(B_0) \subseteq B_1$ and $p$ induces a bijection between $B_1 \setminus i(B_0)$ and $B_2$. Analogously, we define \emph{short based exact sequences} of finite based free $\ZZ G$-chain complexes.
\end{defin}

\begin{defin}
    An \emph{additive $L^2$-torsion invariant} is a map defined on all finite based free $L^2$-acyclic $\ZZ G$-chain complexes, with values in an abelian group, such that:
    \begin{itemize}
        \item if $0 \to C_\blt \to D_\blt \to E_\blt \to 0$ is short based exact, then $a(D_\blt) = a(C_\blt) + a(E_\blt)$;
        \item $a(\dots \xrightarrow{} 0 \xrightarrow{} \ZZ G \xrightarrow{\pm \id} \ZZ G \xrightarrow{} 0) = 0$.
    \end{itemize}
\end{defin}

\begin{thm}[{\cite[Remark~1.16]{uni-l2}}]\label{thm:univ-l2-univ}
    The universal $L^2$-torsion $\rho^{(2)}_u$ is a $\til{K}_1^w(\ZZ G)$-valued additive $L^2$-torsion invariant. Moreover, any additive $L^2$-torsion invariant uniquely factors through $\rho^{(2)}_u$.
\end{thm}

In fact, the twisted $L^2$-Euler characteristic is also an additive $L^2$-torsion invariant with values in the group of set-maps $\Hom(G, \ZZ) \to \ZZ$. Therefore, we can expect to be able to write it in terms of $\rho^{(2)}_u$, possibly with the help of the Dieudonné determinant.

\begin{lemma}\label{lemma:det-k1}
The Dieudonné determinant $\det \colon \mathrm{GL}(n, \linn(G)) \to \linn(G)^\times_\mathrm{ab}$ extends to the weak $K_1$ group as a map
\[
    \det \colon K_1^w(\ZZ G) \to \linn(G)^\times_\mathrm{ab}.
\]
\end{lemma}
\begin{proof}
    By Remark~\ref{rmk:weak-iso-DG}, the weak $K_1$ group is generated by $\linn(G)$-invertible matrices, which admit a Dieudonné determinant. We can extend $\det$ $\ZZ$-linearly, provided that the relations in Definition~\ref{def:weak-k1} are respected. Relation (1) amounts to multiplicativity of the determinant, while (2) requires that
    \begin{equation}
        \det \begin{bmatrix} A & C \\ 0 & B \end{bmatrix} = \det A \det B.
    \end{equation}
    This is a well-known property of the Dieudonné determinant: to prove it, one can nullify $C$ with elementary moves, exploiting invertibility of $A$ or $B$, and then take the Bruhat normal form.
\end{proof}

The universal $L^2$-torsion is not quite an element of the weak $K_1$ group, being defined up to sign, but we still obtain an invariant
\begin{equation}
    \det \rho^{(2)}_u(C_\blt) \in \linn(G)^\times_\mathrm{ab} / \{\pm 1\}.
\end{equation}
Since the sign is inconsequential in the definition of the noncommutative Newton polytope, we can define the \emph{$L^2$-torsion polytope}
\begin{equation}
    P(C_\blt; G) \defeq P_\mathrm{ab}(\det \rho^{(2)}_u(C_\blt)) \in \poly(\ab(G)).
\end{equation}

\subsection{What about manifolds?}
Up to now, this section has remained relatively abstract, focusing on chain complexes instead of $G$-CW complexes. Given a finite free $L^2$-acyclic $G$-CW complex $X$, the main hurdle is that the cellular $\ZZ G$-modules $C_n(X)$ do not have a canonical basis. However, if we choose a representative $n$-cell for each $G$-orbit, we always get the same basis, up to reordering and multiplying every basis element $b_i$ by some $\pm g_i$ for $g_i \in G$.

As the ambiguity mirrors the definition of the weak Whitehead group, it is not hard to see that there is a well-defined universal $L^2$-torsion for $G$-CW complexes:
\begin{equation}
    \rho^{(2)}_u(X) \defeq [\rho^{(2)}_u(C_\blt(X))] \in \mathrm{Wh}^w(\ZZ G).
\end{equation}

Its determinant is defined up to multiplication by signed elements of $G$. Such elements are \emph{monomials} in the skew polynomial ring $\linn(K)[u_1^{\pm 1}, \dots, u_r^{\pm 1}]$, so they act on Newton polytopes by translation. Accordingly, we will extend the notation for the polytope homomorphism:
\begin{equation}
    P_\nu \colon \linn(G)^\times_\mathrm{ab} / \left\langle\begin{bmatrix}\pm g\end{bmatrix}\mid g \in G \right\rangle \to \poly_T(H),
\end{equation}
where $\nu \colon G \surj H$ is a free abelian quotient and $\poly_T(H)$ is the Grothendieck group of integral polytopes of $H$ \emph{up to translation}. This leads to the \emph{$L^2$-torsion polytope of a $G$-CW complex}
\begin{equation}
    P(X; G, \nu) \defeq P_\nu(\det \rho^{(2)}_u(X)) \in \poly_T(H).
\end{equation}
When $\nu = \ab \colon G \surj \ab(G)$, we will simply write $P(X; G)$.

We will now formalize the polytope thickness function of Section~\ref{sec:polyhom} as the \emph{seminorm homomorphism}.
\begin{defin}
    An integral polytope $P \in \polym(H)$ induces a seminorm on $\Hom(H, \ZZ)$:
    \[
        \norm{\phi}_P \defeq \max \{\phi(p) - \phi(p') \mid p, p' \in P\}.
    \]
\end{defin}
\begin{defin}
    The above construction extends to the integral polytope group as the \emph{seminorm homomorphism}
    \[
        \mathfrak N \colon \poly(H) \to \{\Hom(G, \ZZ) \to \ZZ\}, \heqsep
            \mathfrak N(P-Q)(\phi) \defeq \norm{\phi}_{P}-\norm{\phi}_{Q}.
    \]
    Since $\mathfrak N$ is translation invariant, it is well defined as a homomorphism from $\poly_T(H)$, also called $\mathfrak N$.
\end{defin}
\begin{thm}[compare {\cite[Theorem~3.52]{funke}}]\label{thm:diff-seminorms}
    Let $X$ be a finite free $L^2$-acyclic $G$-CW complex and let $\nu\colon G\surj H$ be a free abelian quotient. If $\phi$ is a primitive character that factors through $\nu$, then $X$ is $\phi$-$L^2$-finite and
    \[
        \chi^{(2)}(X; \mathcal N(G), \phi) = \mathfrak N(P_\nu(X; G))(\phi).
    \]
\end{thm}
Like we just said, the seminorm $\norm{\phi}_P$ measures the thickness of the polytope $P$ in the direction $\phi$. When $P$ is a noncommutative Newton polytope, this amounts to the $\phi$-degree of the underlying polynomial. Therefore, we get
\begin{equation}\label{eq:chi-degdet}
     \chi^{(2)}(X; \mathcal N(G), \phi) = \deg_\phi \det \rho^{(2)}_u(X),
\end{equation}
where the degree is, of course, considered as a function on $\linn(G)^\times_\mathrm{ab} / \left\langle\begin{bmatrix}\pm g\end{bmatrix}\mid g \in G \right\rangle$.

Recall now Proposition~\ref{prop:altsum-lapl}, where the adjoint of the universal $L^2$-torsion appears. By~\cite[Lemma~3.18]{uni-l2}, the involution $* \colon \til{K}_1^w(\ZZ G) \to \til{K}_1^w(\ZZ G)$ corresponds to reflecting polytopes with respect to the origin, and hence has no effect on the associated seminorm. We can summarize this as:
\begin{prop}\label{prop:altsum-lapl2}
    Let $X$ be a finite free $L^2$-acyclic $G$-CW complex and let $\Delta_n$ be the combinatorial Laplacians of its $\ZZ G$-chain complex. Then
    \[
        \chi^{(2)}(X; \mathcal N(G), \phi) = \deg_\phi \det \rho_u^{(2)}(X)
        = -\frac 12 \sum_{n\ge 0} (-1)^n \cdot n \cdot \deg_\phi \det \Delta_n.
    \]
\end{prop}

If we have an $L^2$-acyclic compact manifold $M$, we can take $X = \til M$ as a $G$-CW complex, with $G = \pi_1(M)$. While there is no unique way to do so, all such $G$-CW structures give $\ZZ G$-chain homotopic cellular chain complexes, and the $L^2$-torsion polytope is invariant under $\ZZ G$-chain homotopies~\cite{l2poly-inv}.
Hence, we can use Proposition~\ref{prop:altsum-lapl2} to compute this invariant for manifolds using any $G$-CW structure on $\til M$. 

\section{Computing the twisted \texorpdfstring{$L^2$-Euler}{L\^2-Euler} characteristic}
It is time to tie up loose ends and outline a procedure for the computation of the twisted $\chi^{(2)}$ for universal coverings. In what follows, $M$ will be a finite $L^2$-acyclic CW complex or, more commonly, a manifold with the homotopy type of one. Examples include:
\begin{itemize}
    \item closed hyperbolic manifolds of odd dimension~\cite{hyp-l2-acyclic};
    \item manifolds that fiber over the circle, such as complements of fibered links.
\end{itemize}
We will also assume that $G \defeq \pi_1(M)$ is residually finite.

The algorithm starts by computing the $\ZZ G$-chain complex of the universal cover of $M$. Of course, we also take a character $\phi \colon G \to \ZZ$ as input; by homogeneity of the twisted $L^2$-Euler characteristic, we can assume that $\phi$ is primitive, with kernel $K_\phi$.

Next, we compute the Laplacians of the chain complex: indeed, thanks to Proposition~\ref{prop:altsum-lapl2}, we only have to find $\deg_\phi \det \Delta_n$ for $n = 0, \dots, \dim M$. However, as is rightly noted in~\cite[Remark~3.20]{uni-l2}, computing the Dieudonné determinant is very hard, due to our lack of a concrete representation for the elements of $\linn(G)$ for general $G$ (see, however, \cite{linnell-free} for free groups and~\cite{linnell-li} for locally indicable groups). Fortunately, recent developments by Oki~\cite{oki} allow us to probe \emph{valuations} of $\det$ without computing it fully, and without so much as leaving the ring $\ZZ G \subset \linn(K)[u_1^{\pm 1}, \dots, u_r^{\pm 1}]$.

\subsection{Degrees of determinants}

Much like commutative fields, skew fields can be equipped with \emph{valuations}:
\begin{defin}\label{def:valuation}
A \emph{valuation} on a skew field $F$ is a map $v \colon F \to \RR \cup \{+\infty\}$ such that:
\begin{enumerate}[label=(\arabic*)]
    \item $v(ab) = v(a) + v(b)$ for all $a,b \in F$;
    \item $v(a+b) \ge \min \{v(a), v(b)\}$ for all $a,b \in F$;
    \item $v(1) = 0$;
    \item $v(0) = +\infty$.
\end{enumerate}
\end{defin}

It is easy to see that any valuation on $F$ is also well defined on $F^\times_\mathrm{ab}$. We will restrict our attention to \emph{discrete valuations}, that is, those with image in $\ZZ \cup \{+\infty\}$. Examples include the $p$-adic valuations on $\QQ$ and the \emph{order} of a skew Laurent polynomial.

\begin{defin}
    Let $F = T^{-1}(D_t[u^{\pm 1}])$ be the Ore quotient skew field of a skew Laurent polynomial ring. The \emph{order} $\ord \colon F \to \ZZ \cup \{+\infty\}$ is defined on a polynomial as the minimum degree of its monomials, and on a rational element $q(u)^{-1}p(u) \in F$ as $\ord p - \ord q$.
\end{defin}

The degree of a skew Laurent polynomial (or polynomial fraction) is not a valuation; however, if $x \in D_t[u^{\pm 1}]$ has  \emph{symmetric} support, then clearly 
    \[\ord x = -2 \deg x.\]
In the case $F = \linn(G)$, we can generalize this fact via the following technical lemma:

\begin{lemma}\label{lemma:order-deg}
    Let $B \in \mathrm{M}(n, \ZZ G)$ be self-adjoint and invertible over $\linn(G)$, and let $x \in \linn(G)^\times$ be a representative of the Dieudonné determinant of $B$. Considering $\linn(G) = T^{-1} (\linn(K_\phi)[u^{\pm 1}])$ with the order valuation $\ord_\phi$, we have $\deg_\phi x = -2 \ord_\phi x$.
\end{lemma}
\begin{proof}
    Recall~\cite[Lemma~3.18]{uni-l2}, which we applied to obtain Proposition~\ref{prop:altsum-lapl2}: adjunction of matrices corresponds to the central symmetry involution $* \colon \poly(\ab(G)) \to \poly(\ab(G))$.
    
    Let $P_\mathrm{ab}(x) = P-Q \in \poly(\ab(G))$. Since $B$ is self-adjoint, we have $P-Q = *P-*Q$ or, equivalently, $P+*Q = *(P+*Q)$. The degree and order of $x$ can be read off the integral polytope of $x$ over $\ZZ$, that is, $P_\phi(x)$. In turn, this is simply obtained by applying to $P-Q$ the linear projection $\pi_\phi \colon \ab(G) \otimes \RR \to \RR$ induced by $\phi$. Such a projection clearly preserves Minkowski sums and the central symmetry involution.

    Let $\pi_\phi(P) = [a,b]$ and $\pi_\phi(Q) = [c,d]$. Then
    \begin{equation}
        \deg_\phi x = (b-a) - (d-c), \heqsep \ord_\phi x = a-c.
    \end{equation}
    However, $\pi_\phi(P + *Q) = [a,b] + *[c,d] = [a-d, b-c]$ is centrally symmetric, so $b-c = d-a$. The initial claim follows from a routine verification.
\end{proof}

We can now introduce Oki's \emph{matrix expansion algorithm}, following~\cite[Sections~6--7]{oki}. Given a square $n \times n$ matrix $A$, with coefficients in $R \defeq D_t[u^{\pm 1}]$ and invertible over $F$, the algorithm computes $\ord \det A$ as the $D$-rank of a certain matrix $A'$ with entries in the coefficient skew field $D$. Moreover, the entries of $A'$ are obtained from the coefficients of the entries of $A$ using only ring operations in $D$ and the automorphism $t \in \operatorname{Aut}(D)$. In the case $F = \linn(G)$, if $A$ has entries in $\ZZ G$, then we can see them as skew polynomials with coefficients in $\ZZ[K_\phi]$. Hence, the matrix $A'$ has entries in $\ZZ[K_\phi]$ and can be constructed with ring operations in $\ZZ G$. (See Subsection~\ref{sec:impl-oki} for more details on this algorithm.)

As a consequence, we can apply Lemma~\ref{lemma:order-deg} and Oki's matrix expansion algorithm to obtain the degree of the determinant of each Laplacian; 
this step reduces our problem to computing the rank of a square $\ZZ[K_\phi]$-matrix $A$ over $\linn(K_\phi)$. It is worth noting that $K_\phi$ is usually not finitely generated, but its subgroup $S$ generated by the elements appearing in $A$ is.

\subsection{Lück's approximation theorem}

At this point, it is necessary to leave the familiar skew field realm and return to von Neumann algebras. Indeed, Theorem~\ref{thm:DG}~(2) applied to the chain complex \[0 \longrightarrow \ZZ K_\phi^N \xlongrightarrow{A} \ZZ K_\phi^N \longrightarrow 0\] easily implies
\begin{equation}
    \rk_{\linn(K_\phi)} A = \dim_{\mathcal N(K_\phi)} \image(r_A),
\end{equation}
where $r_A\colon \neu(K_\phi)^N\to \neu(K_\phi)^N$ is right multiplication by $A$.
By~\cite[Theorem~6.29~(2)]{l2}, we can replace $\mathcal N(K_\phi)$ with $\mathcal N(S)$: the induction functor $\neu(K_\phi) \otimes_{\neu(S)} {-}$ is right exact and preserves the cokernel of $r_A$ and its dimension, from which the rank can be computed via subtraction from $N$. Now we can present the titular result of this section.

\begin{thm}[{\cite[Section~2]{luckapprox}}]\label{thm:luck-approx}
Let $\Gamma$ be a countable residually finite group and let $A$ be any matrix over $\ZZ \Gamma$.
Let $\Gamma = \Gamma_0 \ge \Gamma_1 \ge \Gamma_2 \ge \dots$ be a residual chain of $\Gamma$, that is, a sequence of finite index normal subgroups with trivial intersection.

If, for $k \ge 0$, $A_k \in \mathrm{M}(n, \ZZ [\Gamma/\Gamma_k])$ is the image
of $A$ under the canonical projection, then we have
\[
    \dim_{\mathcal N(\Gamma)} \image(r_A) = \lim_{k \to +\infty} \dim_{\mathcal N(\Gamma/\Gamma_k)} \image(r_{A_k}).
\]
\end{thm}

Since the groups $\Gamma/\Gamma_k$ on the right hand side are finite, the von Neumann algebras are just complex group algebras $\CC[\Gamma/\Gamma_k]$. Hence, every $A_k$ is but an endomorphism acting on a complex vector space of dimension $n \cdot [\Gamma:\Gamma_k]$, whose von Neumann rank is given by
\begin{equation}\label{eq:vnrank-fin}
    \dim_{\mathcal N(\Gamma/\Gamma_k)} \image(r_{A_k}) = \frac 1{[\Gamma:\Gamma_k]}\cdot \dim_{\CC} \image(r_{A_k}),
\end{equation}
as a consequence of~\cite[Theorem~6.29,(2)]{l2} applied to the inclusion $1 \inj \Gamma / \Gamma_k$.

In our case, the matrices $A_k$ will have integer coefficients, so they can be represented exactly without rounding errors.

\begin{remark}\label{rmk:residual-chain}
    There is a simple (but extremely inefficient) algorithm that computes a residual chain: enumerate all homomorphisms from $\Gamma$ to larger and larger symmetric groups; since $\Gamma$ is residually finite and every finite quotient appears as a subgroup of a symmetric group, the kernels of these homomorphisms will intersect trivially.
\end{remark}

\subsection{Estimating the error}
The statement of Theorem~\ref{thm:luck-approx}, as written, does not provide an effective rate of convergence, but there is a way to make Lück's theorem quantitative, as proved by Löh and Uschold:

\begin{defin}[{\cite[Definition~6.3]{l2comp}}]\label{def:adapted}
Let $\Gamma$ be a countable residually finite group and let $A \in \mathrm{M}(n, \ZZ \Gamma)$ be self-adjoint. We say that a sequence $(\Gamma_k)_{k \in \mathbb N}$ of finite index normal subgroups of $\Gamma$ is \emph{adapted to $A$} if for all $k \ge 0$,
all diagonal entries of $A^0, A^1, \dots , A^{k^2}$
have support in $\Gamma\setminus \Gamma_k \cup \{e\}$.
\end{defin}

\begin{prop}[compare {\cite[Proposition~6.6]{l2comp}}]\label{prop:quant-luck}
Let $\Gamma, (\Gamma_k)_k, A$ be as in Definition~\ref{def:adapted} and let $A_k$ be the image of $A$ in $\mathrm{M}(n, \ZZ [\Gamma/\Gamma_k])$. We have the estimate
\[
|\dim_{\mathcal N(\Gamma)} \image (r_A) - \dim_{\mathcal N(\Gamma/\Gamma_k)} \image (r_{A_k})|
\le n\cdot \left[ 
\left( 1 - \frac 1{kd} \right)^{k^2} + \frac{\log d}{\log k}
\right],
\]
where $d$ is the operator norm of $r_A \colon \ell^2(\Gamma)^n \to \ell^2(\Gamma)^n$.
\end{prop}
Following~\cite[194]{l2}, we can bound the operator norm by $n\sqrt{2}\cdot \max_{i,j} \lVert A_{ij} \rVert_1$, where $\lVert \sum_{g \in \Gamma} \lambda_g g \rVert_1 \coloneqq \sum_{g \in \Gamma} |\lambda_g|$ is the $\ell^1$-norm on $\ZZ \Gamma$. If $\Gamma$ satisfies the Atiyah conjecture, as our group $S$ does, then the von Neumann rank over $\Gamma$ is an integer. To compute it, we choose $k$ such that the right hand side in Proposition~\ref{prop:quant-luck} is less than $1/2$, and then round the von Neumann rank over $\Gamma / \Gamma_k$, computed as in~(\ref{eq:vnrank-fin}), to the nearest integer.

An adapted sequence can be computed provided that we have a solution for the word problem in $\Gamma$ (that is, $S$): see the proof of~\cite[Lemma~6.4]{l2comp}. Since we represent elements of $S$ as elements of $G$, which is finitely presented and residually finite, the proof goes through in the same way, using the algorithm in~\cite[Theorem~5.3]{miller}. The latter consists of a parallel enumeration of relations of $G$ and homomorphisms to symmetric groups, in order to eventually certify that two elements are respectively equal or different; clearly, this is of theoretical interest only, owing to the overwhelming computational complexity. Moreover, the bound in Proposition~\ref{prop:quant-luck} is unfortunately much too loose, because of the term $1/\log k$. Again, this is no obstacle to experimentation.

\begin{remark}\label{rmk:selfadjoint}
    The requirement that $A$ be self-adjoint in Proposition~\ref{prop:quant-luck} is not generally met by our expanded combinatorial Laplacians. However, as noted in~\cite[Remark~2.2]{l2comp}, we may replace $r_A$ by $r_{AA^*}$ without changing the kernel as a module over $\neu(\Gamma)$, for the same reason as in more elementary contexts: if $f \in \neu (\Gamma)$, then $f f^* = 0 \implies f = 0$. Here we must point out that the \emph{adjoint} induced by the antipode of $\ZZ \Gamma$ is none other than the Hermitian adjoint for operators on $\ell^2(\Gamma)$.
\end{remark}

In conclusion, in order to compute the twisted $L^2$-Euler characteristic, we construct the $\ZZ G$-chain complex of the universal cover and the combinatorial Laplacians, in order to apply Proposition~\ref{prop:altsum-lapl2}. By Lemma~\ref{lemma:order-deg} it suffices to compute $\ord \det \Delta_n$ for each Laplacian $\Delta_n$. This can be done using the matrix expansion algorithm and Lück's theorem, with the error bound of Proposition~\ref{prop:quant-luck}. In the end, the problem reduces to computing the ranks of several matrices with integer entries. Hence, we have proved:
\begin{thm}
    There exists an algorithm that, given a finite $L^2$-acyclic CW complex $M$, such that its fundamental group $G$ is residually finite and satisfies the Atiyah conjecture, and a character $\phi\colon G \to \ZZ$, computes the twisted $L^2$-Euler characteristic $\chi^{(2)}(\til M; \phi)$.
\end{thm}

\section{Implementation details}\label{sec:impl}
An implementation of our algorithm is available on GitHub~\cite{git}. It relies on the \textsf{SageMath} software system~\cite{sagemath} and its interface with \textsf{GAP}~\cite{GAP4}, especially the \textsf{HAP} computational homotopy package~\cite{HAP}. We also use \textsf{Regina}~\cite{regina} and \textsf{SnapPy}~\cite{snappy} for testing purposes.

In what follows, we give further explanations of the internal workings of our algorithm.

\subsection{Chain complex construction}

The \textsf{SageMath} system allows for computations in finitely presented groups and their group algebras. Hence, we can represent the differentials of the $\ZZ G$-chain complex of our space $M$ as matrices over $\ZZ G$. In practice, we carry out computations over $\ZZ F$ or $\QQ F$, where $F$ is a free group with the same number of generators as $G$; this is because some ring operations in $\ZZ G$ cause \textsf{SageMath} to attempt to solve the word problem for $G$.

We start by representing $M$ as a \emph{regular CW complex}, that is, we require that every \emph{closed} cell is injectively included into $M$ via its attaching map. This can be ensured by taking the barycentric subdivision of a triangulation of $M$, which is easily obtained using \textsf{Regina} and \textsf{SnapPy} in the case of $3$-manifolds. Given adjacency information between $(i+1)$- and $i$-cells, we can use the \textsf{HAP} method \textsf{ChainComplexOfUniversalCover} to construct the aforementioned differentials.

\subsection{The matrix expansion algorithm}\label{sec:impl-oki}

After constructing the Laplacians, the next major step is applying Oki's algorithm. Recall the problem from the previous section: given a square matrix $A \in \mathrm{M}(n, R) \cap \mathrm{GL}(n, F)$, where $R \defeq D_t[u^{\pm 1}]$ and $F$ is the Ore quotient skew field of $R$, compute $\ord \det A$. The algorithm involves a sequence of matrices over $D$ of increasing size (whence the name), whose ranks converge to $\ord A$.

\begin{remark}
    In our case, any element $a = \sum_{g \in G} a_g g\in \ZZ G$ can be seen as a skew Laurent polynomial in \textsf{SageMath}, by choosing an element $u \in \phi^{-1}(G)$ and collecting monomials with the same exponent of $u$:
        \begin{equation}
            a = \sum_{d \in \ZZ} \Bigg[ \sum_{g \in \phi^{-1}(d)} a_g \cdot gu^{-d} \Bigg] u^d.
        \end{equation}
    Clearly, every coefficient of the resulting polynomial is in $\ZZ[K_\phi]$.
\end{remark}

First, we assume that $A$ does not contain any monomials with negative exponents. In fact, this in ensured by multiplying rows with appropriate powers of $u$ and keeping track of the total multiplier $u^N$; at the end of the computation, we subtract $N$ to get the correct order. In this way, $A$ can be written as
\begin{equation}
    A = \sum_{d = 0}^\ell A_d u^d,  \heqsep A_d \in \mathrm{M}(n, D).
\end{equation}
Consequently, for all $i \ge 0$,
\begin{align}
    u^i A &= \sum_{d = 0}^\ell u^i A_d u^d
\\  &= \sum_{d = 0}^\ell (u^i A_d u^{-i})\cdot u^{d+i}
\\  &= \sum_{d = i}^{\ell+i} (u^i A_{d-i} u^{-i})\cdot u^d,
\end{align}
where $A_d^{(i)} \defeq u^i A_{d-i} u^{-i} \in \mathrm{M}(n,D)$, since conjugation by $u$ is an automorphism of $D$. 

Now, choose an integer \emph{expansion parameter} $\mu \ge 1$ and construct the block matrix
\begin{equation}
    \Omega_\mu(A) \defeq
    \begin{bNiceMatrix}
        A_0^{(0)} & A_1^{(0)} & \dots & \dots & \dots & A_{\mu-1}^{(0)} \\
                0 & A_1^{(1)} & A_2^{(1)} & \dots & \dots & A_{\mu-1}^{(1)} \\
            \vdots &     \ddots & \ddots & \ddots & & \vdots \\
            \vdots &           & \ddots & \ddots & \ddots & \vdots \\
            \vdots &           &       & \ddots & \ddots & A_{\mu-1}^{(\mu-2)} \\
                0 &     \dots & \dots & \dots &     0 & A_{\mu-1}^{(\mu-1)}
    \end{bNiceMatrix}
    \in \mathrm{M}(\mu n, D).
\end{equation}
Lastly, define
\begin{equation}
    \omega_\mu(A) \defeq \rk \Omega_\mu(A),
    \heqsep \psi_\mu(A) \defeq \mu n - \omega_\mu(A).
\end{equation}

\begin{prop}\label{prop:concave-mu}
    The sequence $(\psi_\mu(A))_\mu$ is non-decreasing, concave (i.e.\ its first differences are non-increasing) and eventually equal to $\ord \det A$. In particular, if $M$ is an upper bound on $\ord \det A$, then the limit value is attained by $\mu = M$.
\end{prop}
\begin{proof}
    Firstly, we have $\psi_{d+1}(A) - \psi_d(A) = n - (\omega_{d+1}(A) - \omega_d(A))$. By~\cite[(36)+(33)]{oki}, the difference $\omega_{d+1}(A) - \omega_d(A)$ can be expressed as a quantity $N_d \le n$ non-decreasing in $d$. Therefore, $(\psi_\mu(A))_\mu$ is a non-decreasing concave sequence.

    The second part follows directly from~\cite[Lemma~6.5]{oki}. 
\end{proof}
Importantly,~\cite[Proposition~7.1]{oki} gives an effective upper bound $M = \ell n$, where $\ell$ is the maximum exponent of $u$ appearing in $A$. Hence, we have an algorithm to reduce the computation of $\ord \det A$ to a rank computation over $D$.

The \emph{matrix expansion algorithm} runs in polynomial time in $n$ and $\ell$, provided that elementary arithmetic operations in the skew field $D$ cost $O(1)$.

\begin{remark}
    Oki's article describes another algorithm, the \emph{combinatorial relaxation algorithm}, which is more efficient if the complexity of matrix multiplication is at least $O(n^{2.25})$. Unfortunately, it requires division in $D$, which is inconvenient in our case $D = \linn(\ker \phi)$. Instead, the above algorithm works entirely inside the group ring $\ZZ G$ up to the rank computation.
\end{remark}

\begin{remark} \label{rmk:mu-separate}
    About the expansion parameter $\mu$, Proposition~\ref{prop:concave-mu} suggests that we run the algorithm at $\mu=1,2,3\dots$ until the valuations of the determinants stop increasing. An even better strategy is to take $\mu = 1,2,2^2,2^3,\dots$, and then optionally perform a binary search to find the smallest $\mu$ at which valuations stop increasing. This value may differ for each Laplacian; however, for the sake of simplicity, our implementation applies the same $\mu$ to all Laplacians.
\end{remark}

\subsection{Finite quotients}\label{sec:finquot}
Since the quantitative version of Lück's approximation theorem cannot be used in practice, our strategy is to compute a residual chain and heuristically infer the limit value of the von Neumann rank from the approximating sequence in Theorem~\ref{thm:luck-approx}. As the efficiency of the simple algorithm in Remark~\ref{rmk:residual-chain} is abysmal, we are urged to find a shortcut.

Looking at the statement of Lück's theorem, it is natural to search for large finite quotients of $S$, in the hope that the kernels of the associated epimorphisms are part of a residual chain. Given a finitely presented group, a prime number $p$ and an integer $c \ge 0$, the \textsf{GAP} function \textsf{EpimorphismPGroup} returns the largest $p$-group quotient having nilpotency class $c$. This can be used to construct a function \textsf{finite\_quotient}, taking as input a finitely presented group and a finite list of classes $c = (c_2, c_3, c_5, c_7, \dots)$, and returning the product of all $p$-quotients, each of class $c_p$. By applying \textsf{finite\_quotient} to $G$ and then restricting the epimorphism to $S$, we obtain the desired quotient $S \twoheadrightarrow L$.

\begin{remark}
    This method exposes sizable quotients quickly, but it may not produce a residual system of subgroups even if we let $c$ range over all possible classes. Since the quotients we obtain in this way are always nilpotent, some examples are given by groups that are not residually nilpotent, such as the \emph{modular group} $\mathrm{PSL}(2, \ZZ) \simeq \ZZ_2 * \ZZ_3$ or some free-by-cyclic groups~\cite{res-nilp}. On the other hand, it seems that the resulting kernels are deep enough to approximate the von Neumann rank in every case we examined.
\end{remark}

\begin{remark}\label{rmk:l2-untwisted}
    In order to compute the classical $L^2$-Betti numbers, we can skip the matrix expansion step and use \textsf{finite\_quotient} as a practical way to generate large finite quotients for Lück's approximation theorem. This gives the von Neumann ranks of the differentials, from which the $L^2$-Betti numbers can be obtained via the rank-nullity formula.
\end{remark}

It is also worth noting that the process of Remark~\ref{rmk:mu-separate} becomes heuristic if this method is used, as the former depends on the quality of the Lück approximation step: roughly speaking, if $L$ is large enough, the computed rational value will be close to an integer, which can be inferred to be ``underlying'' value of the order valuation. Hence, in general, manual inspection may be required.

This rounding strategy works whenever the error in the approximation is less than $1/2$. However, intuition based on commutative rings suggests that the rank decreases when passing to a finite approximation (while the valuation increases); see also Funke and Kielak's \emph{determinant comparison problem}~\cite[Question~3.6]{kielak-hab}. Hence, if a stronger heuristic is necessary, we can round up the computed ranks, allowing for an error up to $1$. Following~\cite{lewin-groups}, if $G$ is a (locally indicable) \emph{Lewin group}, then this intuition is correct: every subgroup $S < G$ is also Lewin and the rank $\rk_{\linn(S)}$ is an upper bound for $\rk_{\{1\}}$, that is, the rank arising from the trivial quotient of $S$. By~\cite[Corollary~4.3]{lewin-groups}, the inequality holds for every finite quotient of $S$. It is worth noting that locally indicable free-by-cyclic groups, and conjecturally all locally indicable groups, are Lewin~\cite[Theorem~1.1, Conjecture~1]{lewin-groups}.

\subsection{Non-determinism}\label{sec:nondet}

The internal representation of the homomorphism $\phi$ depends on the presentation of G, and ultimately on \textsf{GAP} internal workings, in its treatment of CW complex objects. The \textsf{GAP} method \textsf{edgeToWord} allows one to relate the basis of $\Hom(G, \ZZ)$ in which $\phi$ is expressed to concrete $1$-cycles on a given CW complex; however, the \textsf{GAP} code loses track of edges during simplification steps, which must then be skipped. In most cases, we have chosen to dispense with this refinement, since doing so only affects the shape of the resulting polytope by a linear transformation.

An additional source of uncertainty is in the construction of the triangulations. Indeed, we often rely on \textsf{Regina} and, in dimension $3$, \textsf{SnapPy} to construct and simplify these structures with non-deterministic algorithms. Hence, simpler $G$-CW complex structures may arise from different runs of the routines.

\section{Examples and applications}

In this section we summarize experimental data on $L^2$-invariants while simultaneously showcasing the algorithm's various aspects; all the experiments are also available as \textsf{SageMath} notebooks in the GitHub repository~\cite{git}. All examples under consideration have residually finite fundamental group, and most are known to satisfy the Atiyah conjecture.

For the sake of brevity, we will use an alternate product-like notation $2^{c_2}\cdot 3^{c_3}\cdot 5^{c_5}\cdot \dots$ for the nilpotency class vector $c=(c_2, c_3, c_5, \dots)$. Moreover, we will often represent characters $\phi$ by their coordinates in some fixed basis of $\Hom(G, \ZZ)$, denoted by $\{\mathsf{v}_i\}$ in figures (see Subsection~\ref{sec:nondet}). Reported running times refer to computations carried out on an Intel i9-12900HX CPU.

\subsection{Convergence rate of finite approximation}
To test the convergence given by Lück's approximation theorem, we apply the untwisted algorithm to a \emph{naturally occurring} matrix. More specifically, let $M$ be the census hyperbolic $3$-manifold $\mathrm{v1539(5,1)}$. 
We use \textsf{HAP} to compute the fundamental group
\begin{equation}
    G\defeq \langle a,b,c\mid
    (a^{2}c^{-5}b^{-3})^{4}a^{2}b^{2},
    a^{2}c^{-5}b^{-2}c,
    c^{9}a^{-3}b^{3}c^{-3}\rangle
\end{equation}
and a $\ZZ G$-chain complex for $\widetilde M$, of the form
\begin{equation}
\begin{tikzcd}
\ZZ G \arrow[r] & \ZZ G^3 \arrow[r, "A"] & \ZZ G^3 \arrow[r] & \ZZ G.
\end{tikzcd}
\end{equation}
Recall that by~\cite[Theorem~1.62]{l2} the above chain complex is $L^2$-acyclic. Therefore, a routine calculation shows that the von Neumann ranks of the three differentials must be $1, 2, 1$. We can test our algorithm on the second differential $A$ with various nilpotency class vectors $c$, expecting it to approximate $\rk A = 2$ (Table~\ref{tab:mat-rank}).

Empirically, the algorithm spends most of its time in the final rank computation over $\ZZ$ or $\QQ$; this suggests a time complexity superquadratic in $|L|$, which is experimentally confirmed. It is therefore essential to keep $|L|$ as low as possible.

Looking at $p$-group quotients, we see that when $p \ne 5$, the error is exactly $1/|L|$; this suggests a trend that holds for all but a finite number of \emph{special} primes depending on the group $G$. Indeed, choosing $c = 5$, $5^2$ and even $2\cdot 3 \cdot 5$ severely degrades the accuracy, even incurring a large computational cost in the latter case.
This behavior can possibly be traced to the matrix $A$ involving only finitely many group elements; we also speculate that the special primes are related to the set of primes $p$ for which $G$ is not residually a (finite) $p$-group.

It also seems that combining two or more non-special primes could be detrimental, as in the case $c = 2\cdot 3$. Hence, we will generally try single primes first, in order to find the special primes, and only then use different primes together.

\begin{table}
\begin{center}
\begin{tabular}{lrS[table-format=1.5]rr} \toprule
    $c$ & $|L|$ & {Output $r$} & $|L|\cdot (2-r)$ & {Time (ms)} \\ \midrule
$2$   &      $4$&   1.75&      1 &       62 \\
$3$   &      $9$&   1.88889&   1 &       32 \\
$5$   &      $25$&   1.64&      9 &       40 \\
$7$   &      $49$&   1.97959&   1 &       55 \\
$11$  &     $121$&   1.99174&   1 &      140 \\
$13$  &   $169$&   1.99408&   1 &      203 \\
$17$  & $289$&   1.99654&   1 &      551 \\ \midrule
$2^2$ &              $32$&   1.96875&   1 &       75 \\
$3^2$ &           $243$&   1.99588&   1 &      609 \\
$5^2$ &        $3125$&   1.99072&  29 &  133 111 \\ \midrule
$2\cdot 3$ &            $36$&   1.86111&   5 &       40 \\
$3\cdot 7$ &       $441$&   1.99773&   1 &    1 150 \\
$2\cdot 3 \cdot 5$ &         $900$&   1.98556&  13 &    4 877 \\
\bottomrule
\end{tabular}
\end{center}
\caption{Rank of $A$ as computed by the algorithm with nilpotency class vector $c$, rounded to five decimal places. Multiplying the error with $|L|$ (smaller is better) reveals an inverse proportionality phenomenon. }\label{tab:mat-rank}
\end{table}

\subsection{Convergence rate of the whole algorithm}
This time, we apply the full algorithm to the $3$-manifold $M$ defined as the complement of the Borromean rings (Thistlethwaite notation L6a4, see Figure~\ref{fig:borromean-link}). Since $M$ is hyperbolic and therefore $L^2$-acyclic, we can compute the twisted $L^2$-Euler characteristic of a primitive character and study the dependence of the approximation on the expansion parameter $\mu$.

\begin{figure}[ht]
    \centering
    \scalebox{1.0}{
    \begin{tikzpicture}[use Hobby shortcut,
    pics/arrow/.style={code={\draw[
        line width=0pt,
        {Computer Modern Rightarrow[
            line width=0.8pt,width=1.5ex,length=1ex,#1]}-] 
        (-0.5ex,0) -- (0.5ex,0);
    }},
    pics/arrowr/.style={code={
    \begin{scope}[xscale=-1,yscale=1]
     \draw[
        line width=0pt,
        {Computer Modern Rightarrow[
            line width=0.8pt,width=1.5ex,length=1ex,#1]}-] 
        (-0.5ex,0) -- (0.5ex,0);
    \end{scope}
    }},
    ]
    \begin{knot}[
    clip width = 7,
    scale = 1.2
    ]
    \def\cs{1.1}
    \def\csy{\cs*sqrt(3)*0.5}
    \strand[blue1, very thick]
    ([closed]{1+0.5*\cs},{\csy}) .. ({2+0.5*\cs},{1+\csy}) ..
    ({1+0.5*\cs},{2+\csy}) .. 
    pic[pos=0.667,sloped]{arrow={scale=2,line width=1pt}}({0.5*\cs},{1+\csy}) .. 
    ({1+0.5*\cs},{\csy});
    \strand[blue2, very thick]
    ([closed]1,0) ..
    pic[pos=0,sloped]{arrowr={scale=2,line width=1pt}} (0,1) ..
    (1,2) .. (2,1) .. (1,0);
    \strand[blue3, very thick]
    ([closed]{1+\cs},0) .. ({2+\cs},1) ..
    pic[pos=0.333,sloped]{arrow={scale=2,line width=1pt}}({1+\cs},2) ..
    ({\cs},1) .. ({1+\cs},0);
    \flipcrossings{3,4}
    \node[blue1!70!black] at (0.35,2.5) {\Large $\lambda_1$};
    \node[blue2!75!black] at (1.15,-0.3) {\Large $\lambda_2$};
    \node[blue3!80!black] at (3.2,1.75) {\Large $\lambda_3$};
    \node at (3.46,-0.54) {};
    \node at (-0.36,3.18) {};
    \end{knot}
    \end{tikzpicture}
    }
    \caption{The Borromean rings.}
    \label{fig:borromean-link}
\end{figure}
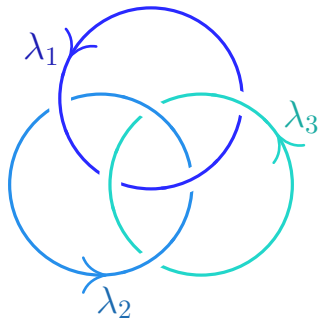

As before, we find a presentation for the fundamental group
\begin{equation}
G \defeq \pi_1(M) = \langle a, b, c \mid c^{-1}b^{-1}cac^{-1}a^{-1}baca^{-1} , b^{-1}abc^{-1}b^{-1}ca^{-1}c^{-1}bc\rangle,
\end{equation}
and a rather simple $\ZZ G$-chain complex of the form
\begin{equation}
\begin{tikzcd}
\ZZ G^2 \arrow[r, "d_2"] & \ZZ G^3 \arrow[r, "d_1"] & \ZZ G.
\end{tikzcd}
\end{equation}

The rank of $G$ is $3$, so we arbitrarily fix a character $\phi$ defined by $(a,b,c) \mapsto (0,0,1)$. First, we try class-$1$ $p$-quotients using $\mu = 6$, for $p$ ranging through the first $10$ primes. In every run the finite quotient has size $|L| = p^2$, and is therefore abelian, for all three Laplacians. Furthermore, the computed valuations $v_i$ and determinant degrees $\delta_i$ are
\begin{equation}
    v_0 = 0,\heqsep v_1 = \frac{12}{p},\heqsep v_2 = \frac{8}{p},\heqsep \delta_0 = 2,\heqsep \delta_1 = 6-\frac{24}{p},\heqsep \delta_2 = 8-\frac{16}{p},
\end{equation}
leading to computed values of $\chi^{(2)}(\til M; \phi) = -1+\frac 4p$. This suggests that all the $v_i$ are in fact $0$, so we could have chosen $\mu=1$, greatly reducing computational time.

We take advantage of this observation and run our code again, with larger quotients and $\mu=1$. Whenever only primes with exponent $1$ are involved, the corresponding $|L|$ is always the squared product of the primes, while the error is proportional to $|L|^{-1/2}$. Further tests show that the coefficient of $|L|^{-1/2}$ for $\mu = 1,2,3,4,5,6$ is respectively $2,4,6,8,6,4$. This suggests that, by choosing $\mu$ as low as possible, we minimize the error (and the computational time due to matrix expansion).
If class-$2$ $p$-quotients are involved, the proportionality holds only approximately, with a coefficient around $10$ for $p=2$ and $20$ for $p=3$.

Hence, the most efficient finite quotients appear to be by far the ones involving only class-$1$ $p$-groups: in this case, $L$ is a subgroup of a product of $1$-nilpotent $p$-groups, and therefore abelian.
This is somewhat surprising, in that kernels of abelian quotients cannot form a residual chain, as they all contain the derived subgroup: effectively, we are approximating the twisted $L^2$-Euler characteristic of the quotient chain complex with coefficients in $\ZZ[\ab(G)]$. In fact, this phenomenon is related to the Alexander norm, which can also be computed from the latter complex, as the following result shows:

\begin{prop}\label{prop:alex-ab}
    Let $M$ be a compact orientable $3$-manifold whose boundary is a union of zero or more tori and whose Alexander polynomial is non-zero. Let $\phi \colon \pi_1(M) \surj \ZZ$ be a primitive character and let $\chi^{(2)}(\til M; \ab(G), \phi)$ be the twisted $L^2$-Euler characteristic of the $\ZZ[\ab(G)]$-chain complex $C_\blt^{\ab}$, i.e.\ the cellular chain complex associated to the $\ab(G)$-covering of $M$. The latter is related to the Alexander norm of $\phi$ via:
    \[
        -\chi^{(2)}(\til M; \ab(G), \phi) = \norm{\phi}_A - \begin{cases}
            0 & \text{if $b_1(M) \ge 2$,} \\
            1 & \text{if $b_1(M) = 1$ and $\partial M \ne \varnothing$,} \\
            2 & \text{if $b_1(M) = 1$ and $\partial M = \varnothing$.} \\
        \end{cases}
    \]
\end{prop}
\begin{proof}
    See Appendix~\ref{sec:alex-ab-proof}; compare also~\cite[Theorem~8.4]{l2thur}.
\end{proof}
Hence, when the Alexander and Thurston norms do not agree (up to the correction term in Proposition~\ref{prop:alex-ab}), using only abelian quotients is destined to fail. For many link complements, including the Borromean rings, equality holds: see~\cite[Section~7]{mcmullen} for a more exhaustive discussion.

\begin{remark}
This result can be generalized to quotients coming from the \emph{rational derived series} $(G_r^{(k)})$ of $G$, provided that the induced $\ZZ[G/G_r^{(k)}]$-chain complex is $L^2$-acyclic. In that case, following~\cite[Section~8]{l2thur}, the twisted $L^2$-Euler characteristic is minus the $k$-th \emph{higher-order Alexander norm} of Harvey~\cite{higher-alex}. (More inequalities between these generalized Alexander norms and the Thurston norm are proved in~\cite[Theorem~2.29]{fbc-bns} and~\cite[Section~6.2]{kielak-sun}.) Our finite quotients are instead approximating the \emph{lower central series}, a similar but distinct construction.
\end{remark}

\begin{remark}
    Compared to the untwisted case, the error decays more slowly in $|L|$, which is the main driver of the running time. We expect this pattern to hold in greater generality.
\end{remark}

\subsection{A few Thurston norm unit balls}
We shall now compute the entire unit balls of the Thurston norms for a few manifolds, including the already mentioned $\mathrm{v1539(5,1)}$ and the Borromean rings complement. This can be done iteratively via Proposition~\ref{prop:thurston-l2}:
\begin{itemize}
    \item compute $-\chi^{(2)}(\til M; \phi)$ for some values of $\phi$;
    \item plot the corresponding points on the unit sphere of the Thurston seminorm;
    \item manually infer a possible shape for the unit ball $B$;
    \item compute $-\chi^{(2)}(\til M; -)$ for all vertices of $B$ and for one point in the interior of each face of $B$, expecting a result of $1$;
    \item if unsuccessful, retry with more points.
\end{itemize}
\subsubsection{The Borromean rings complement}
This is one of the examples in Thurston's original article~\cite{thurston}, where he shows that the unit ball is a regular octahedron up to linear isomorphisms. We will continue to use the $\ZZ G$-chain complex we found in the previous subsection.

\begin{remark}
    The following non-rigorous argument can be used to choose the finite quotient. We heuristically assume that the error in the computation of $\chi^{(2)}(\til M; \phi)$ is bounded above by $20\cdot |L|^{-1/2}$, where $|L|$ is the largest quotient pertaining to any of the three Laplacians. Suppose that the algorithm outputs a value $m'$ within $1/4$ of an integer $m$; then $m$ is the unique integer within distance $3/4$ of $m'$. By the heuristic error bound, we want $3/4 > 20\cdot |L|^{-1/2}$, that is, $|L| \ge 712$; if this is the case, we round the output to $m$. We use the same argument to round the individual valuations $v_i$ to the nearest integers.
\end{remark}

We can accomplish this with $c = 29$; of course, choosing abelian quotients is only acceptable to the extent that the Thurston norm equals the Alexander norm for this manifold. According to the discussion in the previous subsection, it is crucial to choose $\mu$ as small as possible (as in Remark~\ref{rmk:mu-separate}), both for speed and to create more leeway in the error bound. We start by applying the algorithm to a few values of $\phi$ with $\mu=1,2$ (Table~\ref{tab:borromean-ball-1}), finding that the 
rounded valuations do not change between $\mu=1,2$; hence, we deduce that $\mu=1$ suffices. 

\begin{table}
\begin{center}
\begin{tabular}{ccccc} \toprule
    $\phi$ & {$-\chi^{(2)}(\til M; \phi)$} \\
    \midrule
    $(1,0,0)$ & 1 \\
    $(0,1,0)$ & 1 \\
    $(0,0,1)$ & 1 \\
    \bottomrule
\end{tabular}
\end{center}
\caption{A few trials of the algorithm with $c=29$. Each computed value is rounded to the unique integer within distance $1/4$. The maximum $|L|$ is always $29^2 > 712$. Total computational time was around $1$ minute.}\label{tab:borromean-ball-1}
\end{table}

Therefore, the three basis vectors all have Thurston norm $1$; by symmetry, we get six lattice points on the unit sphere. Thurston's argument for this manifold shows that the only six unit norm lattice points are the vertices of the (octahedral) unit ball, represented by the three link components and their opposites. Hence, expecting to find eight points of norm $3$ at $(\pm 1, \pm 1, \pm 1)$, we run the algorithm once again (see Table~\ref{tab:borromean-ball-2}).

\begin{table}
\begin{center}
\begin{tabular}{cccccc} \toprule
    $\phi$ & $\mu$ & {$v_0$} & \multicolumn{1}{c}{$v_1$} & \multicolumn{1}{c}{$v_2$} & \multicolumn{1}{c}{$-\chi^{(2)}(\til M; \phi)$} \\ \midrule
    $(1,1,1)$  & 1 & 0 & 1 & 0 & 3 \\
    $(-1,1,1)$ & 2 & 0 & 2 & 0 & 3 \\
    $(1,-1,1)$ & 2 & 0 & 2 & 0 & 3 \\
    $(1,1,-1)$ & 2 & 0 & 2 & 0 & 3 \\
    \bottomrule
\end{tabular}
\end{center}
\caption{Four lattice points taken from the cones over open faces of the octahedron. Outputs are already integer numbers and $|L|$ is always $29^2$. We report the smallest $\mu$ at which Oki's algorithm attains its limit value. Total computational time was around $3$ minutes.}\label{tab:borromean-ball-2}
\end{table}
Again by symmetry, we get that each point in $\{(\pm 1/3, \pm 1/3, \pm 1/3)\}$ (lying in the interior of a face of the octahedron) has unit norm. This already determines the unit ball completely (Figure~\ref{fig:borromean-ball}); the shape is consistent with further runs on the remaining lattice points with coordinates in $\{-1,0,1\}$.
\begin{figure}
    \centering
    \scalebox{0.8}{
    \hspace{-4pt}\includegraphics{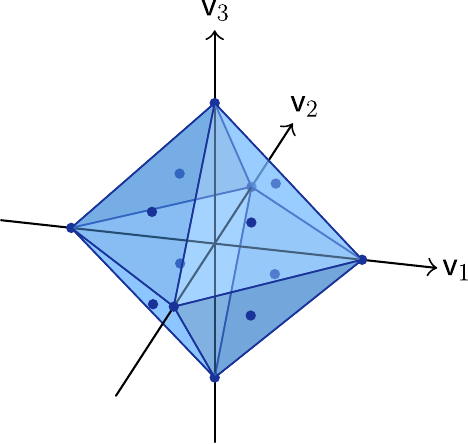}%
    }
    \caption{The unit ball is determined by the $14$ marked points.}
    \label{fig:borromean-ball}
\end{figure}

\subsubsection{The L10n14 link complement}
Next, we test the algorithm on the complement $M$ of the two-component link L10n14 (Figure~\ref{fig:L10n14-balls}), introduced by Dunfield~\cite{dunfield} as a counterexample to the equality of Thurston and Alexander unit ball faces.

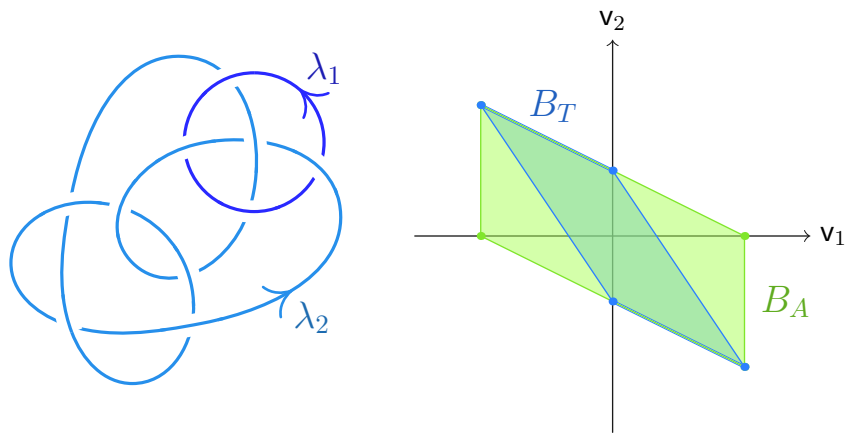
\begin{figure}
    \centering
    \begin{tabular}{cc}
        \raisebox{-0.5\height}{
            \scalebox{1}{
            \begin{tikzpicture}[use Hobby shortcut,
                pics/arrow/.style={code={\draw[
                    line width=0pt,
                    {Computer Modern Rightarrow[
                        line width=0.8pt,width=1.5ex,length=1ex,#1]}-] 
                    (-0.5ex,0) -- (0.5ex,0);
                }},
                pics/arrowr/.style={code={
                \begin{scope}[xscale=-1,yscale=1]
                 \draw[
                    line width=0pt,
                    {Computer Modern Rightarrow[
                        line width=0.8pt,width=1.5ex,length=1ex,#1]}-] 
                    (-0.5ex,0) -- (0.5ex,0);
                \end{scope}
                }},
            ]
            
            \begin{knot}[
            consider self intersections,
            clip width = 7,
            scale = 0.4,
            end tolerance = 0.1pt
            ]
            \def\eps{0.3}
            \strand[blue2, very thick]
            ([closed]2,{2-\eps}) ..
            (2,{6+\eps}) ..
            (6,10.8) ..
            (7,{10+\eps}) ..
            (8,{8+\eps}) ..
            ({8-0.25*\eps},{6+0.25*\eps}) ..
            (4,4) ..
            ({6+0.25*\eps},{8-0.25*\eps}) ..
            ({8+\eps},8) ..
            ({10+\eps},7) ..
            (10.8,6) ..
            pic[pos=0.6,sloped]{arrowr={scale=2,line width=1pt}}({6+\eps},2) ..
            ({2-\eps},2) ..
            (0,4) ..
            ({2+\eps},6) ..
            (5,5) ..
            (6,{2+\eps}) ..
            (4,0) ..
            (2,{2-\eps});
            \strand[blue1, very thick]
            (8,{10+\eps}) arc (90:450:{2+\eps}) pic[pos=0.875,sloped]{arrow={scale=2,line width=1pt}};
            \flipcrossings{2,3,5,6}
            \node[blue2!75!black] at (9.9,2.2) {\Large $\lambda_2$};
            \node[blue1!70!black] at (10.3,10.45) {\Large $\lambda_1$};
            \node at ({-0.5*\eps},{-0.5*\eps}) {};
            \end{knot}
            \end{tikzpicture}
            }
        } & 
        \raisebox{-0.5\height}{
            \scalebox{1}{
            \hspace{-4pt}\includegraphics{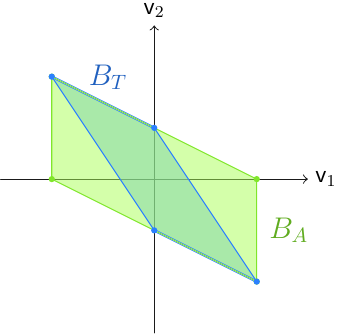}%
            }
        }
    \end{tabular}
    \caption{The two-component link L10n14 and the unit balls for its Thurston and Alexander norms.}
    \label{fig:L10n14-balls}
\end{figure}

As always, we first compute a presentation for the fundamental group
\begin{equation}
    G \defeq \langle a,b \mid a(ba^{-1}b^{-2}a^{-1})^2ba^3ba^{-1}b^{-2}a^{-1}ba^{-1}b^{-1}(ab^2ab^{-1})^2a^{-3}b^{-1}ab^2ab^{-1} \rangle
\end{equation}
and a $\ZZ G$-chain complex, of the form
\begin{equation}
\begin{tikzcd}
\ZZ G \arrow[r, "d_2"] & \ZZ G^2 \arrow[r, "d_1"] & \ZZ G.
\end{tikzcd}
\end{equation}

The presentation for $G$ is the same as the one given directly by \textsf{SnapPy}, up to a cyclic permutation of the relator. Hence we can access information about the link meridians, obtaining that the cohomology classes of the components $\lambda_1, \lambda_2$ are respectively $(1,-1)$ and $(-1,0)$ in our basis. We can also compute the Thurston norm using an algorithm by Friedl, Tillmann and Schreve for two-generator one-relator groups~\cite{thurst-fox}, verifying the strict inclusion between the two norm balls (again Figure~\ref{fig:L10n14-balls}).

Specifically, on a few selected points, the Thurston and Alexander norms take on the values in Table~\ref{tab:l10n14-norms}, which completely determine the Thurston norm.
The last three values are shared by the Alexander norm; indeed, for $\phi = (0,1),(-1,2),(-1,1)$ our algorithm gives consistent results across several finite quotients, both abelian and non-abelian.

\begin{table}
    \centering
    \begin{tabular}{ccc}\toprule
        $\phi$ & $x_M(\phi)$ & $\norm{\phi}_A$ \\ \midrule
         $(1,0)$ & 3 & 1 \\
         $(0,1)$ & 2 & 2 \\
        $(-1,2)$ & 3 & 3 \\
        $(-1,1)$ & 1 & 1 \\ \bottomrule
    \end{tabular}
    \caption{A few points in $H^1(M; \ZZ)$ and their Thurston and Alexander norms.}
    \label{tab:l10n14-norms}
\end{table}

On the other hand, the two norms differ for $\phi = (1,0)$: accordingly, we expect the Thurston norm to be harder to approximate, requiring highly non-abelian quotients. As even $c = 2^2$ gives an abelian $L \simeq \ZZ_2 \times \ZZ_4$, we anticipate a very demanding computation, due to the fast growth of the order of $L$ with respect to the nilpotency class.

In the end, it is hard to get very close to the true value of the Thurston norm, even with $30$ minutes of computational time; however, the results (collected in Table~\ref{tab:l10n14-hard}) are still consistent with Table~\ref{tab:l10n14-norms}.

\begin{table}
    \centering
    \begin{tabular}{lrS[table-format=1.0]S[table-format=1.3]S[table-format=1.3]S[table-format=1.0]S[table-format=1.3]S[table-format=1.3]S[table-format=1.3]S[table-format=4.2]}\toprule
        $c$ & $|L|$ & {$v_0$} & {$v_1$} & {$v_2$}
        & {$\delta_0$} & {$\delta_1$} & {$\delta_2$}
        & {$-\chi^{(2)}(\til M; \phi)$} & {Time (s)} \\ \midrule
        $2$ & 2    & 0 & 4 & 2 & 2 & 6 & 4 & 1 & 0.23 \\
        $2^2$ & 8    & 0 & 4 & 2 & 2 & 6 & 4 & 1 & 0.17 \\
        $2^3$ & 128  & 0 & 3 & 1 & 2 & 8 & 6 & 2 & 0.51 \\
    $3^3$   & 2187 & 0 & 2.667 & 0.667 & 2 & 8.667 & 6.667 & 2.333 & 76 \\
        $2^4$ & 8192 & 0 & 2.5 & 0.5 & 2 & 9 & 7 & 2.5 & 1726 \\ \midrule
        \emph{limit}
            & {--} & 0 & 2 & 0 & 2 & 10 & 8 & 3 & {--} \\
        \bottomrule
    \end{tabular}
    \caption{Results obtained from the point $\phi = (1,0)$ with $\mu = 2$, including the inferred limit values. Outputs are rounded to three decimal places.}
    \label{tab:l10n14-hard}
\end{table}

\subsubsection{The census manifold \texorpdfstring{$\mathrm{v1539(5,1)}$}{v1539(5,1)}}

This manifold is unique in the orientable closed census for having first Betti number greater than one (in fact, equal to $2$). Unlike our first example, we will use a simpler chain complex
\begin{equation}
\begin{tikzcd}
\ZZ G \arrow[r, "d_3"] & \ZZ G^2 \arrow[r, "d_2"] & \ZZ G^2 \arrow[r, "d_1"] & \ZZ G,
\end{tikzcd}
\end{equation}
with
\begin{equation}\label{eq:v1539-group}
    G\defeq \langle a,b\mid
        b^{-4}a^2(b^{-1}a^{-1})^3 b^{-1} a^2 b^5 a^2 b^{-1} a^{-3} b^4, (ba^{-2} ba^3)^2 b a^{-2} b^{-5} \rangle.
\end{equation}

It turns out that taking $c = 3^2$ is sufficient to obtain exact integer outputs, provided $\mu$ is as large as necessary for the valuations to stop changing: the minimal values of $\mu$ are listed along with the outputs in Table~\ref{tab:v1539-outputs}. In more detail, we first try the vectors $\phi = (1,0), (0,1), (1,1), (1,-1)$ (in a basis dual to the generators $a,b$), which give $8$ points on the boundary of the unit ball (refer to Figure~\ref{fig:v1539-balls}, left). There is still some ambiguity, resolved by the input $\phi = (2,-1)$, which gives the top-left and bottom-right corners of the unit ball.

\begin{table}
    \centering
    \begin{tabular}{crrrrrc}\toprule
        $\phi$ & $\mu$ & {$\delta_0$} & {$\delta_1$} & {$\delta_2$} & {$\delta_3$} & {$-\chi^{(2)}(\til M; \phi)$} \\ \midrule
        $(1,0)$  &  4 & 2 & 10 & 10 & 2 & 2 \\
        $(0,1)$  & 10 & 2 & 12 & 14 & 4 & 2 \\
        $(1,1)$  & 11 & 2 & 18 & 22 & 6 & 4 \\
        $(1,-1)$ & 16 & 2 & 10 & 10 & 2 & 2 \\
        $(2,-1)$ & 18 & 4 & 14 & 12 & 2 & 2 \\
        \bottomrule
    \end{tabular}
    \caption{Computed values of the Thurston norm of various classes with $c = 3^2$. Total running time was under $10$ seconds.}
    \label{tab:v1539-outputs}
\end{table}

\begin{figure}[ht]
    \centering
    \begin{tabular}{cc}
        \raisebox{-0.5\height}{
            \scalebox{1.25}{
            \hspace{-4pt}\includegraphics{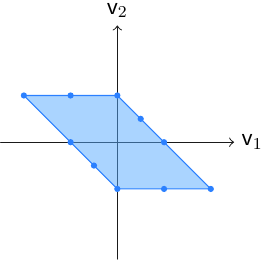}}
        } & 
        \raisebox{-0.5\height}{
            \scalebox{1.25}{
            \hspace{-4pt}\includegraphics{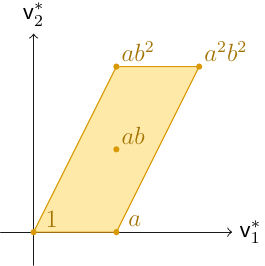}}
        }
    \end{tabular}
    \caption{Left: the Thurston norm unit ball for $\mathrm{v1539(5,1)}$, determined by the $10$ marked points. Right: the Newton polytope for the Alexander polynomial obtained from the group presentation~(\ref{eq:v1539-group}). The two polytopes are dual to each other.}
    \label{fig:v1539-balls}
\end{figure}

For many choices of $c$, even involving higher-class quotients, the group $L$ turns out to be abelian. Fortunately, this is not a problem: by the discussion in~\cite[Example~2.1]{button}, \emph{every} non-vertex cohomology class of $\mathrm{v1539(5,1)}$ is fibered, so the Thurston norm equals the Alexander norm.
Of course, the Alexander polynomial can be easily computed from the group presentation~(\ref{eq:v1539-group}):
\begin{equation}
    \Delta_M = 1+a+ab+ab^2+a^2b^2 \in \ZZ[a^{\pm 1}, b^{\pm 1}].
\end{equation}
We observe that its Newton polytope (Figure~\ref{fig:v1539-balls}, right) is dual to the norm ball we computed, providing further confirmation.

\subsubsection{Manifolds with \texorpdfstring{rank-$1$}{rank-1} first homology}
Whenever the first Betti number of a manifold is $1$, its Thurston norm becomes a simple, integer-valued invariant, requiring in principle no pattern recognition in order to determine unit ball vertices. If we restrict our attention to closed $3$-manifolds, McMullen's inequality gives
\begin{equation}\label{eq:mcmullen-betti1}
    \norm{\phi}_A \le x_M(\phi) + 2,     
\end{equation}
where $\phi$ is a generator of $H^1(M; \ZZ)\simeq \ZZ$.

\begin{table}[p!]
    \centering
    \begin{tabular}{l@{}r} \toprule
    \multicolumn{1}{c}{$M$} & \multicolumn{1}{c}{$x_M$} \\ \midrule
 m160(3,1)          & 2 \\
 m159(4,1)          & 2 \\
 m199(-4,1)         & 2 \\
 m122(-4,1)         & 2 \\
 s942(-2,1)         & 2 \\
 m336(-1,3)         & 2 \\
 m345(1,2)          & 4 \\
 m289(7,1)          & 2 \\
 m280(1,4)          & 2 \\
 m304(-5,1)         & 2 \\
 m305(-1,3)         & 2 \\
 s385(5,1)          & 4 \\
 s296(-1,3)         & 2 \\
 s297(5,1)          & 2 \\
 s912(0,1)          & 2 \\
 m401(-2,3)         & 2 \\
 m371(-1,3)         & 2 \\
 m368(-4,1)         & 2 \\
 s580(-5,1)         & 2 \\
 s581(-1,3)         & 2 \\
 s869(-1,2)         & 2 \\
 s861(3,1)          & 2 \\
 v1191(-5,1)        & 2 \\
 v1076(-5,1)        & 2 \\
 \textbf{s528(-1,3)}        & \textbf{2} \\
 \textbf{s527(-5,1)}        & \textbf{2} \\
 s924(3,1)          & 2 \\
 v1408(4,1)         & 2 \\
 s677(1,3)          & 2 \\
 s676(5,1)          & 2 \\
 v2641(-4,1)        & 2 \\
 s745(3,2)          & 2 \\ \bottomrule
 \end{tabular} \;\;
 \begin{tabular}{l@{}r}\toprule
    \multicolumn{1}{c}{$M$} & \multicolumn{1}{c}{$x_M$} \\ \midrule
 \textbf{s644(-4,3)}        & \textbf{2} \\
 \textbf{s643(-5,1)}        & \textbf{2} \\
 s646(5,2)          & 4 \\
 s789(-5,1)         & 2 \\
 s719(7,1)          & 2 \\
 v1373(-2,3)        & 2 \\
 \textbf{v2018(-4,1)}       & \textbf{2} \\
 v3209(3,1)         & 2 \\
 v2420(-3,1)        & 2 \\
 v2099(-4,1)        & 2 \\
 v2101(3,1)         & 2 \\
 s789(5,1)          & 2 \\
 v1539(-5,1)        & 2 \\
 \textbf{v1436(-5,1)}       & \textbf{2} \\
 v1721(1,4)         & 8 \\
 \textbf{s750(4,3)}         & \textbf{2} \\
 \textbf{s749(5,1)}         & \textbf{2} \\
 \textbf{s789(-5,2)}        & \textbf{2} \\
 \textbf{v1539(5,2)}        & \textbf{2} \\
 \textbf{v2238(-5,1)}       & \textbf{2} \\
 \textbf{v3209(1,2)}        & \textbf{2} \\
 \textbf{s828(-4,3)}        & \textbf{2} \\
 \textbf{v1695(5,1)}        & \textbf{2} \\
 v2771(-4,1)        & 4 \\
 s836(-6,1)         & 6 \\
 v2986(1,2)         & 4 \\
 v2209(2,3)         & 4 \\
 \textbf{s862(7,1)}         & \textbf{2} \\
 \textbf{v2190(4,1)}        & \textbf{2} \\
 v2054(-7,1)        & 2 \\
 v3066(-1,2)        & 4 \\
 v2563(5,1)         & 2 \\\bottomrule
 \end{tabular} \;\;
 \begin{tabular}{l@{}r}\toprule
    \multicolumn{1}{c}{$M$} & \multicolumn{1}{c}{$x_M$} \\ \midrule
 v2345(5,1)         & 2 \\
 v3209(-3,1)        & 2 \\ 
 v3077(5,1)         & 2 \\
 v2959(-3,1)        & 2 \\
 v2671(-2,3)        & 4 \\
 \textbf{v3209(-1,2)}       & \textbf{2} \\
 \textbf{v2593(4,1)}        & \textbf{2} \\
 s928(2,3)          & 6 \\
 v3390(3,1)         & 2 \\
 v3209(4,1)         & 2 \\
 v2913(-3,2)        & 2 \\
 v3505(-3,1)        & 2 \\
 v3261(4,1)         & 2 \\
 v3262(3,1)         & 2 \\
 v2678(-5,1)        & 2 \\
 \textbf{v3209(3,2)}        & \textbf{2} \\
 \textbf{v3027(-3,1)}       & \textbf{2} \\
 \textbf{v2896(-6,1)}       & \textbf{2} \\
 \textbf{v2683(-6,1)}       & \textbf{2} \\
 \textbf{v2796(4,1)}        & \textbf{2} \\
 \textbf{v2797(-3,4)}       & \textbf{2} \\
 v3107(3,2)         & 4 \\
 v3216(4,1)         & 2 \\
 v3217(-1,3)        & 2 \\
 v3320(4,1)         & 4 \\
 v3091(-2,3)        & 4 \\
 \textbf{v2948(-6,1)}       & \textbf{2} \\
 \textbf{v2794(-6,1)}       & \textbf{2} \\
 v3214(1,3)         & 2 \\
 v3215(-4,1)        & 2 \\
 \textbf{v3183(-3,2)}       & \textbf{2} \\
 v3209(-4,1)        & 2 \\\bottomrule
 \end{tabular} \;\;
 \begin{tabular}{l@{}r}\toprule
    \multicolumn{1}{c}{$M$} & \multicolumn{1}{c}{$x_M$} \\ \midrule
 v2984(-1,3)        & 4 \\
 \textbf{v3145(3,2)}        & \textbf{2} \\
 \textbf{v3181(-3,2)}       & \textbf{2} \\ 
 v3209(5,1)         & 2 \\
 v3019(5,2)         & 6 \\
 \textbf{v3036(3,2)}        & \textbf{2} \\
 v3212(1,3)         & 4 \\
 \textbf{v3209(1,3)}        & \textbf{2} \\
 \textbf{v3269(4,1)}        & \textbf{2} \\
 \textbf{v3209(-3,2)}       & \textbf{2} \\
 \textbf{v3209(2,3)}        & \textbf{2} \\
 \textbf{v3313(3,1)}        & \textbf{2} \\
 \textbf{v3239(3,2)}        & \textbf{2} \\
 \textbf{v3209(5,2)}        & \textbf{2} \\
 \textbf{v3209(-1,3)}       & \textbf{2} \\
 v3209(-5,1)        & 2 \\
 v3425(-3,2)        & 4 \\
 v3209(6,1)         & 2 \\
 \textbf{v3209(4,3)}        & \textbf{2} \\
 v3318(4,1)         & 6 \\
 \textbf{v3244(4,3)}        & \textbf{2} \\
 \textbf{v3243(-4,1)}       & \textbf{2} \\
 v3352(1,4)         & 6 \\
 v3398(2,3)         & 4 \\
 v3378(-1,4)        & 6 \\
 v3408(1,3)         & 8 \\
 v3467(-2,3)        & 8 \\
 v3445(6,1)         & 10 \\
 v3509(4,3)         & 4 \\
 v3508(4,1)         & 4 \\
 v3504(-2,3)        & 6 \\ 
 & \\ \bottomrule
    \end{tabular}
    \caption{Thurston norms of closed hyperbolic $3$-manifolds in the census with rank-$1$ first homology. Rows in bold denote non-fibered manifolds~\cite[Table~4]{button}.}
    \label{tab:betti1-thurston-norms}
\end{table}

The hyperbolic $3$-manifold census lists $127$ closed $3$-manifolds with first Betti number equal to $1$, of which $86$ are fibered and $41$ non-fibered~\cite{button}; surprisingly, for all of them, equality holds in~(\ref{eq:mcmullen-betti1}). Indeed, this is part of McMullen's statement for the fibered manifolds, while the rest can all be shown with \textsf{SageMath} to contain appropriate normal surfaces of Euler characteristic $-2$ and to have Alexander polynomials of degree $4$.
Combining this with Proposition~\ref{prop:alex-ab}, we obtain
\begin{equation}\label{eq:betti1-ab-suffices}
    \chi^{(2)}(\til M; \phi) = \chi^{(2)}(\til M; \ab(G), \phi),
\end{equation}
for $M$ any of the $127$ manifolds we are considering.

When we attempt to apply our algorithm to these manifolds, we encounter a limitation of the finite quotient subroutine: nontrivial $p$-group quotients are found only if $H_1(M; \ZZ)$ has $p$-torsion. This appears to be strongly related to $\ker \phi$ having finite, or even trivial, abelianization. However, due to~(\ref{eq:betti1-ab-suffices}), it is not necessary to consider non-abelian quotients: in practice, it turns out that the trivial quotient always suffices, attaining exactly (!) the Thurston norm, if $\mu$ is high enough. The results are listed in Table~\ref{tab:betti1-thurston-norms}.

\subsection{Free-by-cyclic groups}
Another rich class of test cases is given by \emph{free-by-cyclic} groups:

\begin{defin}
    A \emph{free-by-cyclic group} is a semidirect product $F_n \rtimes_\varphi \ZZ$, where $F_n$ is a free group of finite rank $n \ge 0$ and $\varphi$ is an automorphism of $F_n$. Equivalently, it is an extension of a finitely generated free group by the infinite cyclic group $\ZZ$. 
\end{defin}

Every compact $3$-manifold with boundary that fibers over $S^1$ has a free-by-cyclic fundamental group: the fiber is a compact surface with boundary and so has free fundamental group, while $\varphi$ is simply the monodromy of the fibration. Thus, it is natural to consider such groups as generalizations of fibered $3$-manifolds. Following~\cite[Definition~6.50]{l2} and~\cite[Section~4.1]{uni-l2}, we define the twisted $L^2$-Euler characteristic of a group $G$ as that of the universal cover of its classifying space $BG$, if the latter is finite and $L^2$-acyclic.

A free-by-cyclic group $G = F_n \rtimes_\varphi \ZZ$ has a canonical presentation
\begin{equation}
    \langle x_1, \dots, x_n, t \mid tx_i t^{-1} = \varphi(x_i)\ \forall i\in \{1,\dots,n\} \rangle,
\end{equation}
which is \emph{combinatorially aspherical} as defined in~\cite{aspherical-pres}, by the Corollary of Theorem~3.4 in the same article. To obtain a model for $BG$, we further prove that the presentation complex $K$ is aspherical. By~\cite[Proposition~1.3]{aspherical-pres}, it suffices to check that:
\begin{itemize}
    \item[--] no relation $r_i \defeq tx_i t^{-1} \varphi(x_i)^{-1}$ is a proper power;
    \item[--] no two relations are freely conjugate to each other or their inverses.
\end{itemize}
The former follows from the relations being cyclically reduced and containing only one occurrence of the \emph{stable letter} $t$; to prove the latter, note that each relation (and its inverse) has a unique cyclically reduced representative starting with $t$, and that the second letter is $x_i$ for $r_i$ and $x_i^{-1}$ for $r_i^{-1}$.

At this point, it remains to show that our model for the classifying space is $L^2$-acyclic. This can be done by applying either Theorem~1.39 or Theorem~7.2~(5) in~\cite{l2}.
As a bonus, free-by-cyclic groups satisfy the Atiyah Conjecture by~\cite[Theorem~10.22]{l2}, and are residually finite by~\cite{fbc-resfin}.

Constructing test cases becomes a matter of finding free group automorphisms, which can be done by composing elements from a generating set of $\operatorname{Aut}(F_n)$, such as the following involutions inspired by~\cite{aut-fn}:
\begin{itemize}
    \item[--] $\tau_i$, inverting the generator $x_i$;
    \item[--] $\sigma_{i,j}$, swapping $x_i$ and $x_j$;
    \item[--] $\eta_{i,j}$, sending $x_i \mapsto x_j^{-1}x_i$, $x_j \mapsto x_j^{-1}$, and leaving other generators unchanged.
\end{itemize}
In order to determine the invariant $-\chi^{(2)}(G; -)$ in a finite number of evaluations, we exploit the fact that it is a genuine seminorm, proved in~\cite[Corollary~3.5]{fbc-bns}, just like the Thurston norm case. This is not strictly necessary, as the Laplacian degrees are themselves seminorms by Kielak's \emph{single polytope theorem}~\cite[Theorem~3.14]{kielak-bns}.

The first example we consider is the automorphism $\varphi \colon F_3 \to F_3$ given by
\begin{equation}
    \varphi \defeq \eta_{2,1}\cdot\sigma_{1,3}\cdot\eta_{2,1}\cdot\eta_{3,2}\cdot\eta_{3,1},
\end{equation}
where the leftmost factor is applied first. We proceed by constructing the necessary \textsf{GAP} objects. Since $G$ has an aspherical group presentation, we can compute the $\ZZ G$-chain complex directly using the \textsf{HAP} method \textsf{ResolutionAsphericalPresentation}. We choose a basis of $\Hom(G, \ZZ) \simeq \ZZ^2$ such that $\phi = (0,1)$ is the canonical projection on $\ZZ$, while $\phi = (1,0)$ has the stable letter $t$ in its kernel. More precisely, the images of the generators of $G$ are:
\vspace{1.5ex}
\begin{equation}
    \label{tab:fbc-basis-1}
    \begin{tabular}{ccccc} 
        \toprule
        {$\phi$} & {$\phi(x_1)$} & {$\phi(x_2)$} & {$\phi(x_3)$} & {$\phi(t)$}
         \\ \midrule
        $(1,0)$ & 1 & 1 & 1 & 0 \\
        $(0,1)$ & 0 & 0 & 0 & 1 \\
        \bottomrule\vspace{-0.75ex}
    \end{tabular}
\end{equation}
Key values of $\chi^{(2)}(G; -)$ are collected in Table~\ref{tab:fbc-results-1}: the two runs are consistent with each other and with the value $2$ for each value of $\phi$ we considered, determining the square unit ball shown in Figure~\ref{fig:fbc-balls}.

\begin{table}
\begin{center}
\begin{tabular}{lccS[table-format=1.5]@{${}\to{}$}cS[table-format=1.5]@{${}\to{}$}cS[table-format=1.5]@{${}\to{}$}c} \toprule
    \multicolumn{1}{l}{$c$} & $\mu$ & \multicolumn{1}{c}{$\phi$} & \multicolumn{2}{c}{$v_1$} & \multicolumn{2}{c}{$v_2$} & \multicolumn{2}{c}{$\hspace{-0.8em}-\chi^{(2)}(G; \phi)$} \\ \midrule
    \multirow{4}{*}{$2^3$} & \multirow{4}{*}{$4$} &     $(0,1)$ & 0 & 0 & 0 & 0 & 2 & 2 \\
        & & $(1,1)$ & 2.4375 & 2 & 1.57812 & 1 & 1.28125 & 2 \\
        & & $(1,0)$ & 1 & 1 & 1 & 1 & 2 & 2 \\
        & & $(-1,1)$ & 3.40625 & 3 & 1.4375 & 1 & 1.53125 & 2 \\ \midrule
    \multirow{4}{*}{$7^2$} & \multirow{4}{*}{$5$} &     $(0,1)$ & 0 & 0 & 0 & 0 & 2 & 2 \\
        & & $(1,1)$ & 2.07872 & 2 & 1.1516 & 1 & 1.77551 & 2 \\
        & & $(1,0)$ & 1 & 1 & 1 & 1 & 2 & 2 \\
        & & $(-1,1)$ & 3 & 3 & 1 & 1 & 2 & 2 \\
    \bottomrule
\end{tabular}
\end{center}
\caption{Experimental results for the twisted $L^2$-Euler characteristic of the first free-by-cyclic example; for simplicity, we omit the values of $v_0$, as they do not affect the result.
}\label{tab:fbc-results-1}
\end{table}

Another (considerably more complex) rank-$2$ test case arises from the automorphism $\varphi: F_4 \to F_4$, defined by
\begin{align}
\begin{split}
    \varphi \defeq{} &\eta_{2,3}\cdot\sigma_{2,3}\cdot\tau_{4}\cdot\eta_{4,3}\cdot\eta_{3,1}\cdot\tau_{3}\cdot\sigma_{3,4}\cdot\sigma_{2,4}\cdot\eta_{3,1}\cdot\tau_{4}\cdot\eta_{2,1}\cdot\sigma_{1,4}\cdot\eta_{1,4} \\
    {}\cdot{} &\eta_{1,4}\cdot\eta_{3,2}\cdot\sigma_{1,4}\cdot\eta_{2,3}\cdot\sigma_{1,3}\cdot\eta_{3,4}\cdot\eta_{3,1}\cdot\eta_{1,2}\cdot\sigma_{3,4}\cdot\eta_{2,4}\cdot\eta_{1,4}\cdot\eta_{3,2}\cdot\eta_{3,1} \\
    {}\cdot{} &\eta_{4,3}\cdot\eta_{2,3}\cdot\eta_{2,3}\cdot\sigma_{1,4}\cdot\sigma_{3,4}\cdot\tau_{2}\cdot\eta_{1,2}\cdot\eta_{1,2}\cdot\sigma_{3,4}\cdot\sigma_{2,3}\cdot\eta_{4,3}\cdot\eta_{1,2}\cdot\eta_{2,1} \\
    {}\cdot{} &\tau_{2}\cdot\eta_{4,3}\cdot\sigma_{2,4}\cdot\sigma_{1,2}\cdot\sigma_{1,2}\cdot\eta_{3,1}\cdot\eta_{4,1}\cdot\sigma_{1,2}\cdot\eta_{3,4}\cdot\sigma_{1,3}\cdot\eta_{2,4}\cdot\eta_{2,4}\cdot\eta_{1,4}.
\end{split}
\end{align}

We fix a basis given by the following table:
\vspace{1.5ex}
\begin{equation}
    \label{tab:fbc-basis-2}
    \begin{tabular}{cccccc} 
        \toprule
        {$\phi$} & {$\phi(x_1)$} & {$\phi(x_2)$} & {$\phi(x_3)$} & {$\phi(x_4)$} & {$\phi(t)$}
         \\ \midrule
        $(1,0)$ & 1 & -2\phantom{-} & -3\phantom{-} & -1\phantom{-} & 0 \\
        $(0,1)$ & 0 & 0 & 0 & 0 & 1 \\
        \bottomrule\vspace{-0.75ex}
    \end{tabular}
\end{equation}
In this case, the algorithm has more trouble converging: when choosing an odd prime $p$, the quotient routine returns cyclic groups, incurring the risk of not approximating the invariant correctly, like in Proposition~\ref{prop:alex-ab}; on the other hand, even the $2$-quotient from $c = 2^3$ is much too large (order $1024$), producing runtimes on the order of hours.

Therefore, we use the improved heuristic from Subsection~\ref{sec:finquot} to infer results from the smaller quotient with $c = 2^2$; that is, we increase the expansion parameter $\mu$ until $\floor*{v_1}$ and $\floor*{v_2}$ stabilize, and use the latter values to compute Laplacian degrees and the twisted $L^2$-Euler characteristic. We can also check that the quotients $L$ are all non-abelian, of the form $(\ZZ_4 \times \ZZ_2) \rtimes \ZZ_4$ or $\ZZ_2 \times (\ZZ_4 \rtimes \ZZ_4)$.

The ten lattice points of Table~\ref{tab:fbc-results-2} outline a decagonal unit ball (Figure~\ref{fig:fbc-balls}) when appropriately scaled. In some cases the error in the valuations is very close to $1$, indicating that it could be even larger, which in turn would lead to incorrect inferences about $v_1$ and $v_2$. This is somewhat mitigated by the general consistency of the shape and by a cross-check with $c = 3^2$.

\begin{figure}
    \centering
    \begin{tabular}{cc}
        \raisebox{-0.5\height}{
            \scalebox{1}{
            \hspace{-4pt}\includegraphics{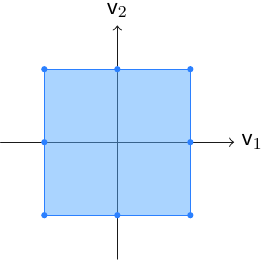}
            }
        } & 
        \raisebox{-0.5\height}{
            \scalebox{1}{
            \hspace{-4pt}\includegraphics{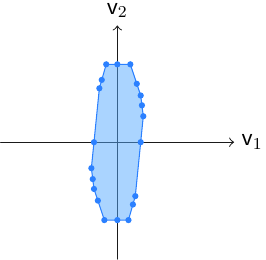}
            }
        }
    \end{tabular}
    \caption{The unit balls of the seminorms $-\chi^{(2)}(G; -)$ for both free-by-cyclic examples, including the boundary points that determine them.}
    \label{fig:fbc-balls}
\end{figure}
\begin{table}
\begin{center}
\begin{tabular}{crS[table-format=2.5]S[table-format=2.5]r} \toprule
    $\phi$ & $\mu$ & \multicolumn{1}{c}{$v_1$} & \multicolumn{1}{c}{$v_2$} & \multicolumn{1}{c}{$\hspace{-0.8em}-\chi^{(2)}(G; \phi)$} \\ \midrule
    $(1,0)$  & $10$ & 30       & 9        & $10$ \\
    $(1,1)$  & $13$ & 32.375   & 10.375   &  $9$ \\
    $(2,3)$  & $23$ & 64.40625 & 19.40625 & $19$ \\
    $(1,2)$  & $11$ & 32       & 9        & $10$ \\
    $(1,3)$  & $13$ & 31.5     & 10.25    & $12$ \\
    $(1,6)$  & $15$ & 24.9375  & 13.9375  & $18$ \\
    $(0,1)$  &  $1$ & 0        & 0        &  $3$ \\
    $(-1,7)$ & $10$ & 7.40625  & 8.40625  & $21$ \\
    $(-1,4)$ &  $9$ & 14       & 17       & $15$ \\
    $(-1,3)$ & $13$ & 18.875   & 20.625   & $13$ \\
    \bottomrule
\end{tabular}
\end{center}
\caption{Results for the second free-by-cyclic example, run with $c = 2^2$. The reported value of $\mu$ is the smallest such that $\lfloor v_1 \rfloor, \lfloor v_2 \rfloor$ coincide for $\mu, \mu+1$; values of $-\chi^{(2)}(G; \phi)$ are computed from $\lfloor v_1 \rfloor, \lfloor v_2 \rfloor$.}\label{tab:fbc-results-2}
\end{table}

In both examples, the class $\phi = (0,1)$ dual to the stable letter may be considered analogous to a fibered class, since it describes the group $G$ as an extension of its finitely generated, free kernel by $\ZZ$, just like in the $3$-manifold case. This can be taken as a definition of fibered class; such matters are discussed in the aforementioned article~\cite{kielak-bns}, which relates them to the \emph{Bieri-Neumann-Strebel} (or \emph{BNS}) \emph{invariant} $\Sigma(G)$. We will limit ourselves to noting that, in both cases, $(0,1)$ lies in the interior of an open cone, and the twisted $L^2$-Euler characteristic agrees with the Euler characteristic of the free group: $\chi(F_n) = 1-n$.

\subsection{The fiber of the Ratcliffe-Tschantz manifold}
In a recent article, Italiano, Martelli and Migliorini~\cite{imm} constructed fibrations over the circle on a few cusped hyperbolic $5$-manifolds of finite volume including, notably, the smallest known hyperbolic $5$-manifold of Ratcliffe and Tschantz $N$. 

Despite the manifold having a relatively simple and concrete combinatorial description, \textsf{GAP} appears to be unable to produce its cellular chain complex, hanging for an extremely long time. Thus, we study its four-dimensional fiber $F$ instead, which is an aspherical manifold of Euler characteristic $1$ admitting no hyperbolic structure.

As always, we construct the $\ZZ G$-chain complex of $\til F$, this time starting from the \textsf{Regina} isomorphism signature given in the paper~\cite[Section~2.3]{imm}, obtaining
\begin{equation}
\begin{tikzcd}
\ZZ G^4 \arrow[r, "d_3"] & \ZZ G^{10} \arrow[r, "d_2"] & \ZZ G^6 \arrow[r, "d_1"] & \ZZ G.
\end{tikzcd}
\end{equation}
Since $G^\mathrm{ab} \simeq \ZZ_4^4$, we expect no nontrivial $p$-quotients for $p \ne 2$: this is supported by a few quick tests. Hence, we try to approximate the von Neumann rank of the three differentials using only $2$-quotients of $G$. In practice, setting $c = 2^2$ in the algorithm provides good results (Table~\ref{tab:fiber}).

\begin{table}
    \centering
    \begin{tabular}{lrS[table-format=4.5]S[table-format=4.5]S[table-format=4.5]}\toprule
        $c$ & {$|L|$} & {$\rk(d_1)$} & {$\rk(d_2)$} & {$\rk(d_3)$}
         \\ \midrule
        $2$ & 16 & 0.9375 & 4.125 & 2.8125 \\
        $2^2$ & 4096 & 0.99976 & 4.98682 & 3.84399 \\ 
        \midrule
        \emph{limit} & {--} & 1 & 5 & 4 \\
        \bottomrule
    \end{tabular}
    \caption{Ranks of the differentials of $C_\blt(\til F)$ with inferred limit values. The computations for $c = 2^2$ took respectively around $500$, $4500$ and $900$ seconds.}
    \label{tab:fiber}
\end{table}
Note that the ranks are \emph{as high as possible}: indeed, $d_1$ and $d_3$ are full rank, while $\rk(d_2)$ is bounded above by $6-\rk(d_1)$. We easily recover the $L^2$-Betti numbers of the fiber via the rank-nullity-type formula
\begin{equation}
    b^{(2)}_i(\til F) = \dim C_i(\til F) - \rk(d_i) - \rk(d_{i+1}),
\end{equation}
obtaining 
\begin{equation}
b^{(2)}_i(\til F) = \begin{cases}
    1 & \text{if $i = 2$,}
\\  0 & \text{otherwise.}
\end{cases}
\end{equation}
This confirms $\chi(F) = 1$ and is reminiscent of the \emph{Singer conjecture} for closed aspherical manifolds:
\begin{conj}[Singer] If $M$ is a closed aspherical $n$-manifold, then all its $L^2$-Betti numbers vanish, except possibly for $b^{(2)}_{n/2}(\til M)$ if $n$ is even.
In that case, if $M$ also admits a metric with negative sectional curvature, then $b^{(2)}_{n/2}(\til M) > 0$.
\end{conj}
This conjecture is known to hold for closed hyperbolic manifolds~\cite{hyp-l2-acyclic}.

\begin{remark}
The computation we carried out suggests a broadening of its scope to include $F$ and possibly some other non-compact manifolds. Let $H$ be an aspherical $n$-manifold which is the interior of a compact manifold $\overline H$ with $L^2$-acyclic boundary, as is the case for $F$ (see \cite[Section~3]{imm} and note that the boundary components are flat $3$-manifolds, hence virtually fibered and $L^2$-acyclic). Following~\cite[2565]{okun-schreve}, we use the \emph{Davis reflection group trick} to construct a compact aspherical $n$-manifold $H'$. Moreover, if we assume the Singer conjecture for $H'$ and the \emph{cadim conjecture} in dimension $n-1$ (which can be ensured for $n\le 4$~\cite[Corollary~4.15]{okun-schreve}), we obtain the cadim conjecture for $H$, i.e. $b_i(\til{H}) = 0$ for $i > n/2$. We can strengthen the condition to $i \ne n/2$ as a consequence of Poincaré duality, the $L^2$-homology long exact sequence of $(\overline H, \partial\overline H)$, and $L^2$-acyclicity of $\partial \overline H$; in other words, $H$ satisfies the Singer conjecture.
\end{remark}

\subsubsection{Some musings and conjectures}
A direct consequence of our computation is that the twisted $L^2$-Betti numbers $b^{(2)}_i(\til N; \phi)$ of the Ratcliffe-Tschantz manifold all vanish, except for $b^{(2)}_2(\til N; \phi) = 1$. Of course, a similar argument can be made for any fibered manifold with fiber satisfying the Singer conjecture:
\begin{lemma}
    Assume the Singer conjecture for all even-dimensional closed aspherical manifolds. Let $M$ be a closed aspherical  $(2n+1)$-manifold and let $\phi\colon \pi_1(M) \surj \ZZ$ be a virtually fibered cohomology class (i.e.\ its image in a finite cover represents a fibration over the circle). Then $b_i^{(2)}(\til M; \phi) = 0$ for all $i \ne n$.
\end{lemma}
\begin{proof}
    Since twisted $L^2$-Betti numbers are multiplicative with respect to finite covers, we can assume that $\phi$ is a fibered class. The infinite cyclic cover associated to $\phi$ is homotopically equivalent to the fiber, which must then be aspherical. By the Singer conjecture, its (untwisted) $L^2$-Betti numbers vanish in dimension not equal to $n$, and so do the twisted $L^2$-Betti numbers of $M$.
\end{proof}
Recall that fibering is an open condition in cohomology, so that by assuming the existence of a single virtually fibered class, we obtain a whole open cone on which $b^{(2)}_i(\til M; {-}) = 0$ for $i\ne n$.
If we add a tameness condition, such as requiring that the twisted $L^2$-Betti numbers be seminorms in the variable $\phi$, we may formulate the following conjecture:
\begin{conj}\label{conj:thurston5}
    Let $M$ be a closed aspherical $(2n+1)$-manifold. Then $b_i^{(2)}(\til M; {-}) \equiv 0$ for all $i \ne n$. Moreover, $\chi^{(2)}(\til M; {-}) = (-1)^n\cdot b_n^{(2)}(\til M; {-})$ is a seminorm up to sign.
\end{conj}
Note that the above holds for $n = 0,1$: in the case of the circle, the twisted $L^2$-Betti numbers are just the $L^2$-Betti numbers of the real line (as a $G$-CW complex with trivial $G$), which vanish in dimension other than $0$; as for $3$-manifolds, this follows from~\cite[Theorem~5.5]{l2thur} by taking $\mu = \id_{\pi_1(M)}$. Indeed, the condition $\image(\mu) \cap \ker(\phi) \ne 1$, or equivalently $\pi_1(M) \ne \ZZ$, holds because $M$ is a closed $3$-manifold and a $K(\pi_1(M), 1)$; hence, $H^3(\pi_1(M); \ZZ_2) = \ZZ_2$, while $H^3(\ZZ; \ZZ_2) = H^3(S^1; \ZZ_2) = 0$. Conjecture~\ref{conj:thurston5} would strengthen the link between the Thurston norm and the twisted $L^2$-Euler characteristic. As an example, if the latter invariant has a polytopal unit ball $B$, it becomes natural to ask whether the set of fibered classes is a union of cones over some open faces of $B$, as in the case of $3$-manifolds. 

\subsection{A closed aspherical \texorpdfstring{$5$-manifold}{5-manifold}}
At the time of writing this paper, there is no explicit description of a closed hyperbolic $5$-manifold as a CW complex. Hence, we will settle for an aspherical manifold $M$ defined as the product of two closed hyperbolic manifolds: $M_3 \defeq \mathrm{m160(3,1)}$ and the genus-$2$ surface $S_2$. The first factor is the first entry in Table~\ref{tab:betti1-thurston-norms} and has $S_2$ as a fiber.

By the Künneth formula, the first cohomology of $M$ is the direct sum 
\begin{equation}
    H^1(M; \RR) \simeq H^1(M_3; \RR) \oplus H^1(S_2; \RR) \simeq \RR \oplus \RR^4.
\end{equation}
After fixing a basis for the second factor, we observe that $(1,0,0,0,0)$ is a fibered class, corresponding to the fiber $S_2\times S_2$ of Euler characteristic $4$. Hence,
\begin{equation}\label{eq:m3xs2-fiber}
    \chi^{(2)}(\til M; (t,0,0,0,0)) = 4\abs{t}.
\end{equation}
In fact, the whole function $\chi^{(2)}(\til M; {-})$ can be determined:
\begin{lemma}
    For all $(t,x,y,z,w) \in \ZZ^5$, we have $\chi^{(2)}(\til M; (t,x,y,z,w)) = 4\abs{t}$.
\end{lemma}
\begin{proof}
    First, we prove that every class $(t,x,y,z,w)$ with $t \ne 0$ is fibered; this is the same as exhibiting a non-vanishing closed $1$-form representing the class. Let $\pi_3\colon M \surj M_3, \pi_2\colon M \surj S_2$ be the natural projections. The $M_3$ factor has a closed non-vanishing $1$-form $\omega$ given by the differential of the fibration, multiplied by $t$. This can be extended to all of $M$ by pulling back, defining a closed $1$-form $\alpha \defeq \pi_3^*(\omega) \in \Omega^1(M)$, which represents the class $(t,0,0,0,0)$.
    
    Now choose $\psi \in \Omega^1(S_2)$ representing $(x,y,z,w)$ and extend it similarly to $\beta \defeq \pi_2^*(\psi)$. Clearly, $\alpha+\beta$ represents $(t,x,y,z,w)$ and is non-vanishing: given a point $(p,q) \in M$, we may choose $v \not \in \ker \omega_p$ as $\omega_p \ne 0$, and then $(v,0) \not \in \ker {(\alpha+\beta)_{(p,q)}}$.

    In a neighborhood of a fibering class, the fiber Euler characteristic is a linear function of the class: this is part of~\cite[Theorem~3]{thurston} and applies to any compact oriented manifold. Therefore, the function $\chi^{(2)}(\til M; {-})$ is locally linear at each point with $t \ne 0$. Since $\{t > 0\}$ is connected, there is a single linear functional $f \colon \RR^5 \to \RR$ equal to $\chi^{(2)}(\til M; {-})$ on the upper half-space. 
    
    As the extension of $\chi^{(2)}(\til M; {-})$ to $\RR^5$ is continuous, being a difference of two seminorms (Theorem~\ref{thm:diff-seminorms}), it must equal $f$ on $\{t = 0\}$; being an even function, this gives $f(0,x,y,z,w) = 0$. Finally, by~(\ref{eq:m3xs2-fiber}), we infer that $f$ is four times the projection on the first coordinate.
\end{proof}

Following Subsection~\ref{sec:nondet}, the product structure of $M$ makes it easier to keep track of the basis in which $\phi$ is expressed. Exploiting this, we ran our algorithm using a $\ZZ G$-chain complex of dimensions $(1,7,16,16,7,1)$, on the class $(1,0,0,0,0)$, with parameters $\mu = 28$, $c = 2$, obtaining exactly the correct result $4$ in under $20$ seconds.

\section{Final remarks}
As we have seen, even if the effective bounds for our algorithm are unusable in practice, we are still able to produce a sequence of approximations that appear to converge at an acceptable rate. In this sense, the algorithm may prove useful as an empirical tool for further research into the twisted $L^2$-Euler characteristic. In fact, the latter is arguably the most natural candidate for the problem of extending the Thurston norm construction to higher-dimensional spaces, due to its behavior with respect to finite coverings and fibrations over the circle. Ironically, however, there seems to be a glaring contrast between the concrete and geometric \emph{minimal complexity} definition of the Thurston norm and the abstract objects involved in $L^2$-invariant theory.

While computational experiments may prove difficult in high dimension due to the sheer combinatorial complexity of the spaces involved, dimension $5$ is probably still accessible, even more so with an optimized version of the algorithm. In particular, if $M$ is a closed aspherical $5$-manifold, we expect results consistent with $\chi^{(2)}(\til M; {-})$ being a seminorm, as in Conjecture~\ref{conj:thurston5}, bridging the gap with the Thurston norm.

Of course, as noted above, more research is also needed on the theoretical front: a geometric interpretation of the twisted $L^2$-Euler characteristic, possibly in a framework that generalizes the Thurston norm, would be incredibly illuminating on these matters.

The Singer conjecture, from which we took inspiration for Conjecture~\ref{conj:thurston5}, might play an important role, as it determines the sign of even-dimensional closed aspherical manifolds, and this sign is inherited by the twisted $L^2$-Euler characteristic of virtually fibered classes in odd-dimensional closed aspherical manifolds. This ``sign constancy'' is exploited in the definition of the Thurston norm by considering only surfaces of non-positive characteristic, which are exactly the aspherical ones.

While our ``algebraic'' approach allows for relative ease of computation, relying on well-established computational homology methods, such a geometric description could produce a more efficient, direct algorithm for the computation of this invariant, of a combinatorial nature akin to normal surface methods in $3$- and $4$-manifold theory. This would, in turn, enable more experimentation and a better understanding of the twisted $L^2$-Euler characteristic.

\clearpage

\appendix
\section{The Alexander norm and the \texorpdfstring{$\ZZ[\ab(G)]$-chain}{Z[ab(G)]-chain} complex}
\label{sec:alex-ab-proof}

In this section we give an overview of the relationship between the Alexander norm and the twisted $L^2$-Euler characteristic of the $\ZZ[\ab(G)]$-chain complex $C_\blt^\mathrm{ab}$ of a compact orientable $3$-manifold, ultimately proving Proposition~\ref{prop:alex-ab}. We assume that the boundary of $M$ is a union of zero or more tori. We refer to McMullen's article~\cite{mcmullen} for an overview of the Alexander norm, recalling from it the definition of the quantity
\begin{equation}
    p \defeq
    \begin{cases}
        0 & \text{if $b_1(M) = 1$,} \\
        1 & \text{if $b_1(M) \ge 2$ and $\partial M \ne \varnothing$,} \\
        2 & \text{if $b_1(M) \ge 2$ and $\partial M = \varnothing$.} \\
    \end{cases}
\end{equation}

By the proof of~\cite[Theorem~5.1]{mcmullen}, the Alexander module of $G$ has a presentation as the cokernel of the second differential $d_2$ of $C_\blt^\mathrm{ab}$. Let $r$ be the rank of $G$. There are a few cases, mirroring the proof we just mentioned.

If $M$ is closed, we can arrange for $C_\blt^\mathrm{ab}$ to be of the form
\begin{equation}\label{eq:alex-zabg-complex}
\begin{tikzcd}
\ZZ[\ab(G)] \arrow[r, "d_3"] & \ZZ[\ab(G)]^n \arrow[r, "d_2"] & \ZZ[\ab(G)]^n \arrow[r, "d_1"] & \ZZ[\ab(G)],
\end{tikzcd}
\end{equation}
where $d_1^T = d_3 = \begin{bmatrix}
    1-g_1 & 1-g_2 & \dots & 1-g_n
\end{bmatrix}$, $\{g_1, \dots, g_{r}\}$ is any basis of $\ab(G)$ and $g_{r + 1} = \dots = g_n = 1$. Since $\phi \colon \ab(G) \surj \ZZ$ has a section, we can choose a basis where $\phi(g_1) = 1$ and $g_2, \dots, g_n \in \ker \phi$; we shall write $t = g_1$.

Moreover, let $d_2'$ be $d_2$ without its first column and $d_2''$ be $d_2$ without its first column and row. Following McMullen's proof, if the rank of $G$ is at least $2$, we have that the multivariate Alexander polynomial $\Delta$ satisfies $\det d_2'' = \pm (1-t)^2 \Delta \ne 0$, so that
\begin{equation}
    \deg_\phi \det d_2'' = 2 + \deg_\phi \det \Delta = 2 + \norm{\phi}_A.
\end{equation}
If instead $r \le 1$, then only the $(1,1)$ minor is nonzero and $\det d_2'' = \pm \Delta$; that is to say, $\deg_\phi \det d_2'' = \norm{\phi}_A$. Combining both subcases, we have
\begin{equation}
    \deg_\phi \det d_2'' = \norm{\phi}_A + p.
\end{equation}

As for the twisted $L^2$-Euler characteristic of $C_\blt^\mathrm{ab}$, 
we refer to the proof of~\cite[Theorem~5.1]{l2thur}, where a chain complex such as (\ref{eq:alex-zabg-complex}) is decomposed using two short exact sequences. First, we quotient $C_\blt^\mathrm{ab}$ by the image of the inclusion of
\begin{equation}
    \mathllap{A_\blt:}\qquad \begin{tikzcd}
\ZZ[\ab(G)] \arrow[r, "1-g_1"] & \ZZ[\ab(G)] \arrow[r] & 0 \arrow[r] & 0,
\end{tikzcd}
\end{equation}
obtaining a complex
\begin{equation}
\begin{tikzcd}
0 \arrow[r] & \ZZ[\ab(G)]^{n-1} \arrow[r, "d_2'"] & \ZZ[\ab(G)]^n \arrow[r, "d_1"] & \ZZ[\ab(G)],
\end{tikzcd}
\end{equation}
where $d_2'$ is $d_2$ without its first column. This complex has a natural projection onto
\begin{equation}
    \mathllap{B_\blt:}\qquad \begin{tikzcd}
 0 \arrow[r] & 0 \arrow[r] & \ZZ[\ab(G)] \arrow[r, "1-g_1"] & \ZZ[\ab(G)],
\end{tikzcd}
\end{equation}
with kernel
\begin{equation}
\begin{tikzcd}
\mathllap{D_\blt:}\qquad 0 \arrow[r] & \ZZ[\ab(G)]^{n-1} \arrow[r, "d_2''"] & \ZZ[\ab(G)]^{n-1} \arrow[r] & 0,
\end{tikzcd}
\end{equation}
where $d_2''$ is $d_2$ without its first column and row. Note that $A_\blt, B_\blt, D_\blt$ are all $L^2$-acyclic, as their non-trivial differentials have non-zero determinant.

Since the universal $L^2$-torsion is additive along based short exact sequences (Theorem~\ref{thm:univ-l2-univ}), we have
\begin{align}
    \chi^{(2)}(\til M; \ab(G), \phi) 
    &= \chi^{(2)}(A_\blt) + \chi^{(2)}(B_\blt) + \chi^{(2)}(D_\blt)
\\ \label{eq:chi-ab-decomp}    &= \deg_\phi \det [1-g_1] - \deg_\phi \det d_2'' + \deg_\phi \det [1-g_1].
\end{align}
The second equality follows from the discussion around~(\ref{eq:chi-el}), with signs added in order to account for the complexes being shifted to the left.

Clearly, the first and last terms in~(\ref{eq:chi-ab-decomp}) are both $1$, while the middle term was computed above; summing up the closed case, we get
\begin{equation}
    \chi^{(2)}(\til M; \ab(G), \phi) = -\norm{\phi}_A + 2-p.
\end{equation}

A similar argument (with only one short exact sequence) shows that if $\partial M \ne \emptyset$, then $\chi^{(2)}(\til M; \ab(G), \phi) = -\norm{\phi}_A + 1-p$. Both cases can be condensed into
\begin{equation}\label{eq:alex-ab-assume}
    \chi^{(2)}(\til M; \ab(G), \phi) = -\norm{\phi}_A +
    \begin{cases}
        0 & \text{if $b_1(M) \ge 2$,} \\
        1 & \text{if $b_1(M) = 1$ and $\partial M \ne \varnothing$,} \\
        2 & \text{if $b_1(M) = 1$ and $\partial M = \varnothing$.} \\
    \end{cases}\vspace{-1.5ex}
\end{equation}
\hfill $\square$

\clearpage

\setlength\bibitemsep{1.5ex}
\printbibliography[heading=bibintoc, title={References}]
\end{document}